\documentclass[a4paper,9pt,reqno]{amsart}
\usepackage{hyperref}
\hypersetup{
    colorlinks,
    citecolor=blue,
    filecolor=blue,
    linkcolor=blue,
    urlcolor=blue
}
\usepackage{amsmath}%
\usepackage{amsfonts}%
\usepackage{amssymb}%
\usepackage{graphicx}
\usepackage{enumitem}
\newtheorem{theorem}{Theorem}[section]
\theoremstyle{plain}

\newtheorem{corollary}[theorem]{Corollary}

\newtheorem{example}[theorem]{Example}

\newtheorem{lemma}[theorem]{Lemma}

\newtheorem{proposition}[theorem]{Proposition}
\newtheorem{remark}{Remark}

\numberwithin{equation}{section}

\newcommand{\R}{\mathbb{R}}
\newcommand{\N}{\mathbb{N}}

\newcommand{\E}{\mathbb{E}}
\newcommand{\Var}{\mathrm{Var}}
\newcommand{\Cov}{\mathrm{Cov}}
\newcommand{\id}{\text{id}}
\newcommand{\LS}{\mathcal{L}}
\newcommand{\Li }{\widetilde{\mathcal{L}}}
\newcommand{\Ti}{\widetilde{\mathcal{T}}}
\newcommand{\T}{\mathcal{T}}
\newcommand{\PIpw}{\mathrm{PI}(p,w)}

\begin{document}

\title[Poincar\'e inequalities by integration]{A note on one-dimensional Poincar\'e inequalities by 
Stein-type integration}
\author{Gilles Germain}
\address[]
{Gilles Germain, Universit\'e libre de Bruxelles, D\'epartement de Math\'ematique, Campus Plaine, Boulevard du Triomphe
CP210, B-1050 Brussels}
\email[]{gilles.germain@ulb.be}%
\author{Yvik Swan}
\address[]{Yvik Swan, Universit\'e libre de Bruxelles, D\'epartement de Math\'ematique, Campus Plaine, Boulevard du Triomphe
CP210, B-1050 Brussels}
\email[]{ yvik.swan@ulb.be}

\begin{abstract}
We study the weighted Poincar\'e constant $C(p,w)$ of a probability density $p$ with weight function $w$ using integration methods inspired by Stein's method. We obtain a new version of the Chen-Wang variational formula which, as a byproduct, yields simple  upper and lower bounds on $C(p,w)$   in terms of the  so-called Stein kernel of $p$.  We also iterate these variational formulas so as to build  sequences of nested intervals containing the Poincar\'e constant, sequences of functions  converging to said constant, as well as sequences of  functions converging to the solutions of the corresponding spectral problem. Our results rely on the properties of a pseudo inverse operator of the classical Sturm-Liouville operator. We illustrate our methods  on a variety of examples: Gaussian functionals,  weighted Gaussian,  beta,  gamma, Subbotin, and Weibull distributions.
\end{abstract}
\maketitle

\section{Introduction and overview of  main results}

Given a measure with density $p$ on the real line, and a weight function
$w >0$, we say that $p$ satisfies a weighted Poincar\'e inequality with
weight $w$ if there exists $C>0$ such that
\begin{equation}
\label{PI}
\Var_p[h]\leq C\,\E_p[|h'|^2w]
\end{equation}
for all $h$ in the Sobolev space $H^1(p,w)$ which we will define in Section \ref{sec:kern-covar-ineq}.  We abbreviate this by
$\mathrm{PI}(p,w)$. The smallest constant for which \eqref{PI} holds
is called the Poincar\'e constant of $p$ with weight $w$, and denoted by
$C(p,w)$. When $C(p,w)< \infty$, we say that a test function $h$ saturates
$\mathrm{PI}(p,w)$ if \eqref{PI} is an equality. Estimation of $C(p,w)$ and of the corresponding saturating functions is  of theoretical importance because of its connections with concentration of measure (see \cite{Ledoux concentration, Gozlan}), isoperimetric inequalities (see \cite{BobkovLedoux, bobhoud}) and quantitative central limit theorems (see \cite{Courtade}). This problem is also of practical importance, with uses  ranging from the study of convergence properties of MCMC algorithms (see \cite{Markov, LyaPoinca, Wang}) to sensitivity analysis \cite{Roustant}. 

Exact solutions are known for some classical measures; for example  $C(p, 1)=1$ for $p$ the standard Gaussian, $C(p, 1)=4$ for $p$ the (double) exponential and $C(p, 1)=(b-a)^2/\pi^2$ for $p$ the uniform on $[a,b]$.  
 A  general bound on $C(p, 1)$ is provided by  the so-called Muckenhoupt criterion (see e.g.\  \cite[Section 4.5.1]{Markov}): letting $P$ be the cumulative distribution function of $p$,
and $\bar{P}=1-P$ the corresponding survival function, it holds that  
\begin{equation} \label{ex:muck}
 B/2       \le C(p, 1) \le 4B,
\end{equation}
where, letting $m$ be a median of $p$,  $B = \max(B_+ ,B_-)$ with $B_+ = \sup_{x > m} \bar{P}(x) \int_m^x 1/p(t) \mathrm{d} t$ and $B_- = \sup_{x < m} {P}(x) \int_x^m 1/p(t) \mathrm{d} t$. In particular, $C(p, 1) < \infty$ if and only if $B < \infty$. Bound \eqref{ex:muck} is, however, not sharp when applied to specific choices of $p$; for instance  it yields  $ 0.239406 \le C(p, 1) \le 1.91525$ in the standard Gaussian case and   $0.5 \le C(p, 1) \le 4$ in the exponential case.
Another elegant universal bound (obtained by a transport argument) is 
\begin{equation} \label{eq:roustbobokov}
    C(p,1) \le 4 \left(\sup_{x \in ]a, b[} \frac{P(x) \bar{P}(x)}{p(x)} \right)^2 \le 4 \left(\sup_{x \in ]a, b[} \frac{\min(P(x), \bar{P}(x))}{p(x)} \right)^2
\end{equation}
where the first inequality is due to  \cite{Roustant}  and the second to \cite{bobhoud}. Again these bounds are generally not  sharp, for instance in the Gaussian case they yield $ C(p, 1) \le \pi/2\approx 1.57$ and in the exponential case $ C(p, 1) \le 4$.

Sharp bounds on $C(p, w)$
can be obtained by couching the problem within the context Sturm-Liouville theory. More precisely, consider  the {Sturm-Liouville operator} for $p$ with weight
$w$ defined for smooth functions $f$ as
\begin{equation}
  \label{eq:SLop}
  \LS f=f''w+f'(-V'w+w')
\end{equation}
where $V = - \ln p$. One important property of this operator is that the inverse of its
smallest non-zero eigenvalue is $C(p,w)$ (see \cite{radial}). In their paper \cite{Chen}, Chen and Wang showed that $C(p,w)$
 satisfies the following variational upper bound   (see \cite[Theorem 2.1]{Chen}):
\begin{equation}
\label{intro2}
\frac{1}{C(p,w)}\geq\inf-\frac{(\LS f)'}{f'}
\end{equation}
for all smooth increasing functions $f$. Note that the authors of \cite{JoulinB} obtained the same formula through a totally different approach linked to intertwining relations (see \cite[Theorem 4.1]{JoulinB}). 
In \cite{Chenalone}, Chen  introduced an iteration of \eqref{intro2} which allows to approximate both from above and from below the weighted Poincar\'e constant of a density on the positive real line to arbitrary precision; this in particular leads to a strict improvement on \eqref{ex:muck}. Chen's method of proof  relies on an operator which we will define in \eqref{eq:chenoperator}. Such iterative schemes also apply in a discrete context, and are of importance for  birth-and-death processes, see e.g.\ \cite{vandoorn} and the many references therein. 

Bounds such as \eqref{ex:muck}, \eqref{eq:roustbobokov}, \eqref{intro2}, and their variations (or iterations) are useful for
obtaining numerical estimates (and sometimes even exact values) for
explicit densities on the real line. Aside from the above references, we also refer to  \cite{JoulinB,Joulin} as well as \cite{Roustant} who  give a nice overview of some literature on the topic and also propose an efficient  algorithm (and \texttt{R}-package) that approximates $C(p,1)$ for any one-dimensional density $p$ with compact support. 

Chen and Wang extend their formula \eqref{intro2} to higher dimensions using a coupling approach (see \cite[Theorem 4.6]{Chen}). Other upper bounds are available in particular cases like uniform and radial measures, convolution of measures or when a Lyapunov function exists (see \cite{uniform, radial, convolution} and \cite{Lyapunov}),  but much remains to be done in this context. However, it is not the object of our paper and we will focus on dimension 1. The extension of our method to higher dimensions seems challenging. 

As mentioned above,  Poincar\'e inequalities have been mostly studied through their relationship with $\LS$. For specific examples of distributions, it may be possible to use classical Sturm-Liouville methods directly to obtain the corresponding solutions.  In this paper, we use another approach, originating in the theory surrounding Stein's method which, as it turns out, nicely generalises the approach from \cite{Chen, Chenalone}. 

Before proceeding to the statement of our results, we first fix some notations.   Let $a\in \R\cup\{-\infty\}$,
$b\in \R\cup\{\infty\}$, and $p\in L^1(]a,b[)$ be such that $p>0$
almost everywhere (a.e.) and the Lebesgue integral of $p$ on the
interval $]a,b[$ is $\int_a^b p=1$.  The function $p$ is thus the
density of a probability measure on the interval $]a,b[$, absolutely
continuous with respect to the Lebesgue measure. A weight is any function $w\in  L^1_{\text{loc}}(]a,b[)$ such that $w>0$ a.e.\ and $pw\in L^1_{\text{loc}}(]a,b[)$.  Throughout, we let
$X$ be distributed according to density $p$ and we write the integral of a
function $f$ with respect to the corresponding measure as
$\E[f(X)] = \E_p[f]=\int_a^b fp = \int_a^bf(x) p(x) \mathrm{d}x$ where the
choice of one notation over the other will be dictated by convenience reasons. 
There is no loss of generality in assuming that $p>0$ a.e. on $]a,b[$. Indeed, if $p=0$ a.e. on $]a,a_1]\cup[b_1,b[$ for some $a<a_1<b_1<b$, we can restrict the domain of $p$ to $]a_1,b_1[$ without changing the value of $C(p,w)$ and if $p=0$ a.e. on an interior interval of $]a,b[$, we must have $C(p,w)=\infty$.

We shall need four operators. The first is the Sturm-Liouville operator defined in \eqref{eq:SLop}. The second is  the canonical density Stein operator $\mathcal{T}$,  defined for a density $p$ on the real line by $$\T h= \frac{(hp)'}{p}$$ for any function $h$ such that $hp$ is weakly differentiable on $]a, b[$. 
This operator is  related to $\LS$ by the identity $\LS h=\T(h'w)$. The third operator is the pseudo-inverse Stein operator $\Ti$ (see \cite{Swan})  defined, for $f \in L^1 (p)$,  by  
\begin{equation*}
\Ti f(x)=\frac{1}{p(x)}\int_a^x (f-\E_p[f])
\end{equation*}
(the denomination pseudo-inverse  will be  explained in Section \ref{sec:kern-covar-ineq}).  Our fourth and final operator is the pseudo-inverse Sturm-Liouville operator  which we define as 
\begin{equation}
\label{Li operator}
\Li f=-\frac{1}{w}\Ti\left(\int_c^{\cdot}f\right),
\end{equation}
for some arbitrary $c\in ]a,b[$ and $f\in L^2(pw)$ (the terminology will be explained in Section \ref{sec:pseudo-inverse-sturm}); this operator can be seen to be equivalent to Chen's operator $II$ from \cite{Chenalone}.

As already stated above, Sturm-Liouville theory, through the spectral properties of $\LS$ and formulas such as \eqref{intro2}, provides a blueprint for (approximately) solving $\mathrm{PI}(p, w)$ with given  $p, w$. Similarly, it is known that properties of $\Ti$ allow to obtain explicit Poincar\'e constants for well chosen weights. Indeed, letting $\mathrm{id}$ be the identity function and setting $\tau = \Ti \mathrm{id}$, we have
\begin{equation}\label{eq:varsaum}
    \Var_p[h]\leq \E_p[|h'|^2 \tau]
\end{equation}
with equality if $h$ is an affine transformation of $\tau$ (see e.g.\  \cite{Saumard, Swan}). The function $\tau$ is called the  Stein kernel of $p$, and  it thus follows in particular that  $C(p, \tau) = 1 $ for any density $p$ with finite variance admitting a Stein kernel.  As already noted in \cite{Swan},  it is not hard to tweak the theory of Stein operators to make appear a connection between ``Stein-type variance bounds'' such as \eqref{eq:varsaum} and ``Chen-Wang-type variance bounds'' such as \eqref{intro2}.  The purpose of our paper is to study this connection in more detail. 

\subsection{Overview of the main results}
\label{sec:overv-main-results}

    In order to make the  paper easier to read and to use, we now  present a streamlined  overview of our main results. We refer to later sections  for more complete statements along with the corresponding proofs, as well as intermediary  side-results that may also be of independent  interest.

 We begin by introducing two kernels whose properties  are at the heart  of  our approach.

\medskip 

 \noindent \textbf{Definition 1.} {\it 
   Let $P = \int_a^{\cdot}p$ be  the cumulative distribution function of $p$,
and $\bar{P}=1-P$ the corresponding survival function. 
  For almost all $x \in ]a, b[$ we define  $$K(x,y)=P\big(x\wedge y\big)\bar{P}\big(x\lor y \big) \mbox{ and } k(x, y) = \frac{K(x, y)}{p(x) p(y)}$$ with  $x \wedge y = \min(x, y)$ and $x \lor y = \max(x, y)$. For any weight $w$, we set 
$k_w(x,y)={k(x,y)}/{(w(x)w(y))}$.   
 }

\medskip

With this notation in hand, our first main result is a  variational formula for Poincar\'e constants (see Theorem \ref{lem1}). 
 \medskip
 
 \noindent \textbf{Theorem 1 (Chen-Wang formula).} {\it For all weakly
 differentiable functions $h_1\in L^1(p)$ and $h_2\in L^2(p)$ which are a.e.\ strictly monotone throughout $]a, b[$ it holds that 
\begin{equation}\label{eq:stequivav1}
\inf_{]a,b[}  \left\{ \frac{\mathbb{E} \left[  k(x,X) h_1'(X) \right]}{w(x) h_1'(x)} \right\}
\leq  C(p,w)\leq \sup_{]a,b[}  \left\{ \frac{\mathbb{E} \left[  k(x,X) h_2'(X)
  \right]}{w(x) h_2'(x)} \right\}
\end{equation}
(the notations $\inf$ and $\sup$ must be understood as the essential infimum and supremum).}

\medskip

Upon closer inspection, the upper bound in \eqref{eq:stequivav1} is seen to be a  equivalent to Chen-Wang's bound  \eqref{intro2}, using \eqref{caracL} (see \cite[Lemma 2.21]{Swan}). The main difference is that in \eqref{eq:stequivav1} we restrict to functions with mean 0, hereby achieving a minor improvement as illustrated in Proposition \ref{our_bound} below (see specifically \eqref{eq:chenwand}). 
As already noted in \cite{Chen}, the freedom of choice in the functions $h_1, h_2$ allows to easily produce non trivial bounds  on $C(p, 1)$. A quite natural choice, related to Muckenhoupt criterion \eqref{ex:muck} and to \cite{Chen, Chenalone}, is $h(x) = \psi(h_0(x))$ where $h_0'(x) =1/(pw)(x)$ and $\psi$ chosen so as to ensure integrability (for instance $\psi(x) = \sqrt x$ suffices, see \cite{Chenalone}). Another interesting choice is 
$h_2'(x) = - p(x)/(P(x) \bar{P}(x))$ (which  always satisfies  $\mathrm{Var}_p[h_2]=\pi^2/3$ irrespective of $p$; we will prove  this curiosity in Appendix \ref{sn:furthproo}). This yields
$$\frac{\mathbb{E} \left[  k(x, X) h_2'(X)
  \right]}{h_2'(x)} = \frac{P(x) \bar{P}(x)}{p(x)^2} \psi(P(x))$$
  with  $\psi(y)=-y\log(y)-(1-y) \log(1-y).$ It follows that 
  \begin{equation}\label{eq:ourroust}
      C(p, w) \le \sup_{x \in ]a, b[}\frac{P(x) \bar{P}(x)}{w(x) p(x)^2} \psi(P(x)). 
  \end{equation}
Although $\psi(u) \le 4 u(1-u)$ over nearly the whole interval $[0,1]$, it can be seen that \eqref{eq:ourroust} does not imply (and is not implied by)  \eqref{eq:roustbobokov}. For instance,  in the case of the Gaussian,  we have already mentioned that the latter yields $C(p,1) \le \pi/2 \approx 1.57$ while \eqref{eq:ourroust} gives $C(p,1) \le \pi\log(\sqrt{2})  \approx 1.09$, whereas for the exponential distribution \eqref{eq:roustbobokov}  yields $C(p, 1) \le 4$ and  \eqref{eq:ourroust}  only produces a trivial bound.  We refer to Section \ref{sec:examples} for more illustrations.

 Another  natural choice in \eqref{eq:stequivav1} is $h_1'=h_2'=-1$. If $p$ has finite second moment,  noting that $\mathbb{E} \left[  k(x,X) 
  \right] = \tau(x)$ is the aforementioned Stein kernel of $p$, it then follows  that 
 \begin{equation}
     \label{eq:stkb} \inf_{]a, b[} \frac{\tau(x)}{w(x)} \le C(p, w) \le  \sup_{]a, b[} \frac{\tau(x)}{w(x)}
 \end{equation}
 (only the upper bound holds if $p$ only has finite first moment). In particular we  immediately read that $C(p, \tau)=1$ when $p$ has finite second moment, hereby confirming \eqref{eq:varsaum}. 
 This already gives nontrivial bounds for densities with bounded Stein kernel; for instance if $p$ is the Gaussian density then $\tau(x) =1$ which leads us back to the known value of the Poincar\'e constant $C(p, 1)$ in this case. Bound \eqref{eq:stkb} is not useful for densities with unbounded Stein kernel, as  e.g.\ for the exponential distribution (in which case $\tau(x) = x$). Again, we refer to Section \ref{sec:examples} for more illustrations.  

Starting from \eqref{eq:stequivav1}, it is intuitively appealing  to iterate the reasoning that leads to Theorem 1 by replacing $h'$ with $\E[k(\cdot,X) h'(X)]/w$; this, as we shall see in Theorem 2 below, indeed leads to sharper bounds on $C(p, w)$. Moreover, since 
\begin{equation} \label{eq:chenoperator}
\Li h'(x)=\frac{1}{p(x)w(x)}\int_a^b K(x,y)h'(y)dy=\frac{1}{w(x)}\E[k(x,X)h'(X)]
\end{equation}
for all weakly differentiable function  $h\in L^1(p)$ (see Lemma 2.21 in \cite{Swan}), we witness how pseudo-inverse Sturm-Liouville  operator introduced in \eqref{Li operator} now comes into play. 
As we shall show in Section~\ref{sec:pseudo-inverse-sturm}, $\Li$ is a continuous, self-adjoint, and positive operator whose  norm is $C(p,w)$.  Let $\Li^{0}f = f$ and define by recurrence $\Li^{n+1} = \Li( \Li^n)$. Replacing iteratively $h'$ by $\Li h'$ in \eqref{eq:stequivav1},  leads to a nested sequence of intervals containing $C(p,w)$, as follows (see Theorem \ref{nested}).

\medskip

\noindent \textbf{Theorem 2 (A sequence of nested intervals).}  Assume that $C(p,w)<\infty$ and $L^2(pw)\subset L^1_{\text{loc}}(]a,b[)$. For all $g_0\in L^2(pw)$ such that $g_0>0$ a.e. the sequence of intervals
\begin{equation*}
I_n=\left[\inf_{]a,b[}\frac{\Li ^{n+1} g_0}{\Li ^n g_0},\sup_{]a,b[}\frac{\Li ^{n+1} g_0}{\Li ^n g_0}\right]
\end{equation*}
satisfies $I_{n+1} \subset I_n$ for all $n\in \N$ and $C(p,w)\in \bigcap_{n\in\N}I_n$.

\medskip

When $a=0$, one readily verifies that \cite[Theorem 1.4]{Chenalone} follows from Theorem 2, by considering $g_0(x) = (pw)(x)^{-1} (\int_0^x 1/(pw))^{-1/2}$.  We note that there
is a priori no guarantee that $\bigcap_{i\in\N}I_n=\left\{C(p,w)\right\}$. For
example, in the case of the exponential measure on $\R^+$ with
$g_0=\id$, we have $I_n=[0,\infty[$ for all $n\in\N$.
In order to obtain convergence results, we have to assume that $\Li$ is compact, which  is true when $k_w\in L^2(pw\otimes pw)$ (see Proposition \ref{bound_k}).  Under this condition, the first eigenvector $e_1$ of $\Li$ is the derivative of the function that saturates $\PIpw$ (see Proposition \ref{prop:eigen}) and we can recover $e_1$ by applying recursively $\Li/C(p,w)$ to any starting function $g_0\in L^2(pw)$. More precisely, the following holds (see Theorem \ref{convergence1}).

\medskip

\noindent \textbf{Theorem 3 (A sequence converging to $e_1$).}
Assume that $C(p,w)<\infty$, $L^2(pw)\subset L^1_{\text{loc}}(]a,b[)$ and $\Li$ is compact. Let $e_1$ be the first eigenfunction of $\Li$, $g_0     \in L^2(pw)$ and $a_1=\E_p[g_0e_1w]$. Then
\begin{equation*}
\frac{\Li^n g_0}{C(p,w)^n}\to a_1e_1,
\end{equation*}
where  the  convergence holds in $L^2(pw)$.

\medskip
In some particular cases,  Theorem 3 provides   $C(p,w)$ as well as  the corresponding saturating function (see Examples \ref{beta exact} and \ref{gamma}). In general, however, one cannot guess $e_1$. The  result remains  useful because it provides  sequences converging to $C(p,w)$, as follows  (see Theorem \ref{convergence2}).

\medskip

\noindent \textbf{Theorem 4 (A sequence converging to $C(p,w)$).}  Assume that $C(p,w)<\infty$, $L^2(pw)\subset L^1_{\text{loc}}(]a,b[)$, and $\Li$ is compact. For all $g_0\in L^2(pw)$, it holds that 
$$\lim_{n\rightarrow \infty}\frac{\Li ^{n+1}g_0(x)}{\Li ^ng_0(x)}=C(p,w)$$
for all  $x \in ]a, b[$
such that $g_0>0$ and $k_{w}(x, \cdot) \in L^2(pw)$.  

\medskip

We conclude this overview of our main results by noting how
\begin{equation*}
\Li^n g_0(x)=\int_a^b\ldots\int_a^b  \frac{K(x,x_1)}{p(x)w(x)}\frac{K(x_1,x_2)}{p(x_1)w(x_1)}\ldots \frac{K(x_{n-1},x_n)}{p(x_{n-1})w(x_{n-1})}g_0(x_n)dx_n\ldots dx_1.
\end{equation*}
If moreover  $\int_a^b pw<\infty$, we can normalize $w$ in order that $pw$ is
the density of a probability measure. Hence, the previous equality can
be reformulated as 
\begin{equation*}
\Li^n g_0(x)=\E\left[k_{w}(x,X_1)k_{w}(X_1,X_2)\ldots k_{w}(X_{n-1},X_n)g_0(X_n)\right]
\end{equation*}
where $X_1,\ldots X_n$ are independent and identically distributed (iid) with density $pw$; this last formula leads, for any reasonable choice of starting function $g_0$,  to easily implemented and  numerically stable approximations of $C(p, w)$.

\subsection{Structure of the paper}
 The rest of the paper is as follows. In Section \ref{sec:preliminaries}, we investigate the properties of the pseudo-inverse Stein  and  pseudo-inverse  Sturm-Liouville operators; most  proofs are given in the Appendix. In Section \ref{sec:main results}, we give detailed statements and  proofs of the theorems presented in the Introduction, along with those of  some secondary results. Most proofs are provided in the text.   Finally,  in Section \ref{sec:examples}, we provide some examples of exact and approximate Poincar\'e constants obtained with our methods; all proofs are provided in the Appendix.  
The supplementary material contains all relevant \texttt{Mathematica}  codes, hereby enabling the interested reader to reproduce our computations.

\section{Preliminaries}
\label{sec:preliminaries}

Most of the  results in  this section  are  extensions  (or particularizations) of  material already available from the literature. In order to keep the paper self-contained (and also because, in some instances, our assumptions are different from those in the literature), we propose bespoke proofs in  Appendix \ref{sn:furthproo}. 

\subsection{The pseudo-inverse Stein operator} \label{sec:kern-covar-ineq}

Let $V = -\ln p$. Following
\cite{Swan}, we define the {canonical Stein operators} for $p$ as
$$\T
f=\frac{(fp)'}{p} = f'- V f \mbox{ and } \Ti h=\frac{1}{p}\int_a^\cdot
(h-\E_p[h])p$$ for a function $f$ such that $fp$ is weakly differentiable on the
one hand, and $h\in L^1(p)$ on the other hand. 
As shown in \cite[Lemma 2.6]{Swan},
$\widetilde{\mathcal T}$ is a pseudo-inverse of $\mathcal T$ in the
sense that $(\Ti \circ \mathcal T )f=f$ if $\E_p[(fp)']=0$
while $(\mathcal T \circ \Ti)h = h-\E_p[h]$. Moreover, if $h$ itself  is
furthermore also weakly differentiable, 
it holds that
\begin{equation}
\label{caracL}
\Ti h(x)=-\frac{1}{p(x)}\int_{a}^b K(x,y)h'(y)dy=-\E\left[k(x,X)h'(X)\right]
\end{equation}
where $X\sim p$ and the kernels $K$ and $k$ above are defined in Definition 1 (see \cite[Lemma 2.21]{Swan}). 
We say that a weakly differentiable function $h$ is increasing if $h'\geq 0$ a.e. The same convention holds for a monotone function.  
We recall that $L^2(p)$ and $L^2(pw)$ are
separable Hilbert spaces (see \cite[Section 3.2 and Proposition
3.5.5]{Heinonen}) and define the Sobolev space
$$H^1(p,w)=\left\{h\in L^2(p): h\text{ is weakly differentiable and
  }h'\in L^2(pw) \right\}.$$ Note how, if $h\in H^1(p,w)$, then both
$h$ and $h'$ belong to $L^1_{\text{loc}}(]a,b[)$. The following holds. 
\begin{lemma}
\label{dif_increasing}
If $C(p,w)<\infty$, every $f\in H^1(p,w)$ can be written $f=f_1-f_2$ for some increasing functions $f_1,\, f_2\in H^1(p,w)$.
\end{lemma}
  One of the most remarkable facts about $\Ti$ is that it appears in the following Hoeffding-type covariance
representation taken from \cite[Corollary 2.4]{SaumardW}.
\begin{theorem}[Hoeffding-type covariance identity]
\label{theo2}
Let $g,h\in L^1(p)$ be weakly differentiable and increasing. Then
\begin{equation}\label{eq:saumaaaa}
  \Cov_p[g,h]=\mathbb{E}_p \left[ g' (- \Ti h)  \right]=\E \left[h'(X)k(X,Y)g'(Y)\right]
\end{equation}
where  
$X,Y\sim p$ are taken independent. These equalities also hold if
$g\in H^1(p,w_1)$ and $h\in H^1(p,w_2)$ for some weights $w_1,w_2$
such that $C(p,w_i)<\infty$ for $i=1,2$.
\end{theorem}
We will only use this result in the case of two functions
in $H^1(p,w)$ for the same $w$, but we emphasize that it holds in a
more general setting. 
 
Using the Hoeffding covariance identity, we can design a set of weights with
finite Poincar\'e constants (see \cite[Corollary 3.6]{Swan}).
For this purpose, we need the next lemma which says that the maximization in $\PIpw$ can be restricted to increasing functions (such a result is not new and can, for instance, be read from \cite{Miclo}).
\begin{lemma}[Monotonicity]
\label{lem increasing}
For all non monotone functions $h\in H^1(p,w)$, we can find an increasing function $g\in H^1(p,w)$ such that
$$
\frac{\Var_p[h]}{\E_p[|h'|^2]}< \frac{\Var_p[g]}{\E_p[|g'|^2]}.
$$
\end{lemma}
\begin{theorem}[Papathanasiou-type upper bound]
\label{theo1}
Let $h\in L^1(p)$ be weakly differentiable such that $h'<0$ a.e. and define $w_h ={-\Ti h}/{h'}$. Then
$C(p,w_h)\leq 1$. In other words,
\begin{equation*}
\Var_p[g]\leq \E_p\left[\frac{-\Ti h}{h'}\left|g'\right|^2\right]
\end{equation*}
for all $g\in H^1(p,w_h)$.  
Equality holds if and only if $h\in L^2(p)$
and $g$ is an affine transformation of $h$. Hence, $C(p,w_h)=1$ if $h\in L^2(p)$.
\end{theorem}

Aside from the requirement of monotonicity, there is near total freedom of choice for  the
function $h$ in Theorem \ref{theo1}; the choice
$h = -\id$ (recall that $\id$ is the identity function) is intuitively a most natural one, and as we
now show, whenever this choice is allowed then it is 
optimal in the following sense. 

\begin{corollary}\label{lem:optimstek}
  Let $X \sim p$. If $\mathrm{Var}[X]<\infty$ then
  $\tau := -\Ti\mathrm{id}$ is the optimal weight with respect to
  the $L^1(p)$ norm, in the sense that
  $ \mathrm{Var}[X] = \| \tau\|_{L^1(p)} \le \| w \|_{L^1(p)}$ for all
  weights $w$ such that $C(p, w) = 1$.
\end{corollary}

The corresponding weight $\tau = -\Ti \id$ is called the
Stein kernel of $p$. It has long been known to be an important handle
on the density $p$ and multivariate extensions are a topic of active
research; see \cite{Saumard, Swan3,Courtade} for more detail and further
references.

\subsection{The pseudo-inverse Sturm-Liouville  operator}
\label{sec:pseudo-inverse-sturm}

In a spirit similar to the Stein operator $\T$ and its pseudo-inverse $\Ti$, we want to define the pseudo-inverse of the Sturm-Liouville operator  $\LS$. First, we need to find a proper domain for this  operator. 
We write $H^1_c(p,w)=\left\{h\in H^1(p,w): \E_p[h]=0\right\}$, with $H^1(p, w)$ as in Section
\ref{sec:kern-covar-ineq}, and we endow it with the scalar
product $(f,h)\mapsto \E_p[f'h'w]$. If $C(p,w)<\infty$, then
$\left\|h\right\|_{L^2(p)}^2\leq C(p,w)
\left\|h'\right\|_{L^2(pw)}^2=C(p,w)\left\|h\right\|_{H^1_c(p,w)}^2$
so that the resulting norm is equivalent to the usual norm
$\left\|h\right\|^2_{H^1(p,w)}=\left\|h\right\|^2_{L^2(p)}+\left\|h'\right\|_{L^2(pw)}^2$.
We define the space of functions
\begin{equation}E^2(p,w)=\left\{f\in L^2(pw)\cap L^1_{\text{loc}}(]a,b[): x\mapsto\int_{x_0}^x f\in L^2(p)\right\}\label{eq:e2pw}
\end{equation}
where $a < x_0< b$ is finite and
arbitrary.
We endow $E^2(p,w)$ with the norm of $L^2(pw)$. Finally we denote the
integral and differential operators by
$$ If=\int_{x_0}^{\cdot}f-\E_p\left[\int_{x_0}^{\cdot}f\right] \mbox{ and
}
Dh=h'$$ for $f\in E^2(p,w)$ and  $h\in H^1(p,w)$, respectively. Note how neither the
definition of $E^2(p,w)$ nor the definition of $I$ depend on the
choice of $x_0\in ]a,b[$.
\begin{proposition}
\label{ID}
The operators $I:E^2(p,w)\rightarrow H^1_c(p,w)$ and
$D:H^1_c(p,w)\rightarrow E^2(p,w)$ are continuous and satisfy $I=D^{-1}$. 
If $C(p,w)<\infty$, then $E^2(p,w)=L^2(pw)\cap L^1_{\text{loc}}(]a,b[)$ 
and the following statements are equivalent :
\begin{enumerate}
\item $E^2(p,w)$ is a Hilbert space,
\item $H^1_c(p,w)$ is a Hilbert space,
\item $L^2(pw)\subset L^1_{\text{loc}}(]a,b[)$,
\item $E^2(p,w)=L^2(pw)$.
\end{enumerate}
\end{proposition}

The inclusion $L^2(pw)\subset L^1_{\text{loc}}(]a,b[)$ is true whenever $pw$ is
continuous and strictly positive on $]a,b[$, which is the case in many
classical examples. 
For the remainder of this section, we assume that $p$ and $w$ are
chosen so that
\begin{itemize}
	\item[(H1)]$C(p,w)<\infty$,
	\item[(H2)]$L^2(pw)\subset L^1_{\text{loc}}(]a,b[)$.
\end{itemize}
Those hypotheses ensure that $E^2(p,w)$ is a Hilbert space by Proposition \ref{ID}. The same proposition says that $E^2(p,w)=L^2(p,w)$, so the notation $E^2(p,w)$ is a bit superficial, which is why we dropped it from the statements of the results in the Introduction. We shall nevertheless continue  using this notation in the current section so as  to emphasize the relation with $H^1_c(p,w)$. 
With these notations, we recall the pseudo-inverse Sturm-Liouville operator $\Li:E^2(p,w)\rightarrow E^2(p,w)$  defined by
\begin{equation*}
\Li f= -\mathcal \Ti I f = \frac{1}{p(x)w(x)}\int_a^b K(x,y)f(y)dy=\frac{1}{w(x)}\E[k(x,X)f(X)]
\end{equation*}
where $X\sim p$. The following then holds. 
\begin{proposition}
\label{Li continuous}
Under (H1)-(H2),  $\Li$ is well-defined, continuous, self-adjoint, positive and satisfies
$\left\|\Li \right\|_{E^2(p,w)\rightarrow E^2(p,w)}=C(p,w).$
\end{proposition}

\begin{proposition}
\label{prop:eigen}
Assume that (H1)-(H2) hold. If $e\in E^2(p,w)$ is an eigenvector of $\Li$, its eigenvalue is $C(p,w)$ if and only if $e>0$ a.e. Further, if such a eigenvector exists, it is unique and $Ie$ saturates $\PIpw$.
\end{proposition}

Another characteristic that may be interesting is compactness, because compact self-adjoint operators have strong spectral properties.
\begin{proposition}
\label{AAA}
Assume that (H1)-(H2) hold. The operator $\Li$ is compact if and only if
\begin{itemize}
	\item[(A1)]the eigenvalues $\left\{\kappa_i:i\in \N_0\right\}$ of $\Li $ verify $\kappa_i>0$, $\kappa_{i+1}\leq \kappa_i$ for all $i\in\N_0$ and $\lim_{i\rightarrow \infty}\kappa_i=0$,
	\item[(A2)]there exists a countable Hilbert basis $\left\{e_i:i\in \N_0\right\}$ of $E^2(p,w)$ made up of eigenvectors of $\Li $.
\end{itemize}
Moreover, if $\Li$ is compact, we have $\kappa_1=C(p,w)$.

\end{proposition}
In the sequel, we won't use explicitly the compactness of $\Li$, but we will refer repeatedly to (A1)-(A2). Our convergence results rely on these properties. 
\begin{example}
Consider the uniform measure on $[0,1]$. From the proof of forthcoming Example \ref{exemple uniforme}, we know that the eigenvectors of $\Li$ are $e_i(x)=\sin((2i-1)\pi x)$ and the eigenvalues are $\kappa_i={1}/({(2i-1)^2\pi^2})$.
\end{example}
We can also show that the converse of Proposition \ref{prop:eigen} holds if $\Li$ is compact.
\begin{proposition} \label{prop:compconditiona}
Assume that (H1)-(H2) hold. If $\Li$ is compact, $h\in H^1_c(p,w)$ saturates $\mathrm{PI}(p,w)$ if and only if $h'$ is an eigenvector of $\Li $ associated to the eigenvalue $\kappa_1=C(p,w)$. 
\end{proposition}
It remains to be seen when $\Li$ is compact. By Proposition \ref{Li continuous}, we know that if $C(p,w)<\infty$ or, equivalently, if the injection of $H^1_c(p,w)$ in $L^2(p)$ is continuous, then $\Li$ is continuous. Actually, the same relation holds for compactness. 
\begin{proposition}
\label{C12}
Assume that (H1)-(H2) hold. If $H^1(p,w)$ is dense in $L^2(p)$ and the injection of $H^1(p,w)$ in $L^2(p)$ is compact, then $\Li$ is compact.
\end{proposition}
The hypotheses of this proposition have already been studied and we can find in the literature more explicit conditions on $p$ under which they hold (see \cite{Hooton} when $w=1$). They mean that $H^1(p,w)$ shares some properties of the classical Sobolev space $H^1(]a,b[)$. In particular, the second one is the counterpart for $H^1(p,w)$ of the Rellich-Kondrachov Theorem (see Theorem IX.16 in \cite{Brezis}). 
Nevertheless, they are uneasy to verify. As $\Li$ is a kernel operator, there is another sufficient condition to check its compactness.
\begin{proposition}
\label{bound_k}
Assume that (H1)-(H2) hold. If $k_{w}\in L^2(pw\otimes pw)$, then  $\Li$ is compact and
\begin{equation*}
C(p,w)^2<  \sum_{i=1}^{\infty}\kappa_i^2 =\left\|k_w\right\|_{L^2(pw\otimes pw)}^2.
\end{equation*}
\end{proposition}

\section{Statements and proofs of the main results}
\label{sec:main results}
\subsection{Variational bounds on $C(p,w)$}

From here onward,    we use the  notations $\inf$ and $\sup$ to denote the  essential
infimum and supremum over $]a, b[$.  Using Theorem~\ref{theo1} we obtain the following.

\begin{theorem}[Chen-Wang variational formula]
\label{lem1}
Let $h_1\in L^1(p)$ and $h_2\in L^2(p)$ be weakly differentiable
and such that $h_1',h_2'<0$ a.e. Then
\begin{equation*}
\inf-\frac{\Ti h_2}{h_2'w}\leq C(p,w)\leq \sup -\frac{\Ti h_1}{h_1'w}.
\end{equation*}
Furthermore, if $-{\Ti h_2}/{(h_2'w)}$ is constant, $h_2$ saturates $\PIpw$. 
\end{theorem}

\begin{proof}[Proof of Theorem \ref{lem1}]
Let $h\in L^1(p)$ be such that $h'<0$ a.e.\ on $]a,b[$ and set $w_{h}:={-\Ti h}/{h'}$. We start with the upper bound. If $\sup w_h/w=\infty$, there is nothing to prove. Assume that $\sup w_h/w<\infty$. As $w_h, w>0$ a.e. we have $\inf w/w_h=(\sup w_h/w)^{-1}>0$. Using this and Theorem \ref{theo1}, we obtain
\begin{equation}
\label{eq:wh}
\E_p\left[|g'|^2w\right]\geq \E_p\left[|g'|^2w_h\right]\inf \frac{w}{w_h}\geq  \Var_p[g]\inf \frac{w}{w_h}
\end{equation}
for all $g\in H^1(p,w)\subset H^1(p,w_h)$. Hence $C(p,w)\leq \sup w_h/w$.

Now, we look at the lower bound. Assume by contradiction that 
$h\in L^2(p)$ and $C(p,w)< \inf {w_{h}}/{w}$. 
  Then there exists some $\epsilon>0$
  such that $C(p,w)+\epsilon \leq {w_{h}}/{w}$ a.e.\ on
  $]a,b[$. Since $w>0$ a.e., we have $(C(p,w)+\epsilon)w \leq w_{h}$
  a.e.\ so that  $H^1(p,w_{h})\subset H^1(p,w)$ and
  $$
  \Var_p[g]\leq C(p,w) \E_p\left[|g'|^2w\right]\leq \frac{C(p,w)}{C(p,w)+\epsilon}\E_p\left[|g'|^2w_{h}\right]
  $$ 
  for all $g\in H^1(p,w_h)$. This is a contradiction since Theorem \ref{theo1} states that $C(p,w_{h})= 1$.
  
Finally, assume that $h\in L^2(p)$ and $w_h/w$ is constant. 
We have $w_h=C(p,w)w$ by the first part of the proof. By Theorem \ref{theo1}, we know that $h$ saturates $\mathrm{PI}(p,w_h)$ and thus $\PIpw$.
\qedhere
\end{proof}

Taking $h=\LS f$ for a smooth function $f$, we obtain a similar result for the Sturm-Liouville operator (see  Appendix \ref{sn:furthproo} for a proof).
\begin{proposition}
\label{our_bound}
Assume that $p, w\in C^2(]a,b[)$. Let
$f_i\in C^{\infty}(]a,b[)$ be such that $-(\LS f_i)'>0$ on $]a,b[$ and
$\T(f_i'w)\in L^i(p)$ for $i=1,2$. We have
\begin{equation*}
\inf\frac{f_2'w-\Phi(f_2'w)}{-(\LS f_2)'w}\leq C(p,w)\leq \sup \frac{f_1'w-\Phi(f_1'w)}{-(\LS f_1)'w}
\end{equation*}
where
$\Phi g(x)={\bar P (x)}/{p(x)}  \lim_{t \to a}g(t)p(t)+ {P
  (x)}/{p(x)} \lim_{t \to b}g(t)p(t)$. 
\end{proposition}
One obvious benefit of Proposition \ref{our_bound} over Theorem \ref{lem1} is that it's easier to take derivatives  than to integrate. 
If we take $f\in C^{\infty}(]a,b[)$ such that $h:=\LS f\in L^1(p)$ and $h'<0$, we have $f'w-\Phi(f'w)=\Ti h>0$ (see the proof of Proposition \ref{our_bound} for details). 
Thus, we can inverse the upper bound of Proposition \ref{our_bound} to obtain 
\begin{equation}\label{eq:chenwand}
\frac{1}{C(p,w)}\geq \inf \frac{-(\LS f)'w}{f'w-\Phi (f'w)}. 
\end{equation}
If we also assume $f'>0$ on $]a,b[$, we have $\Phi(f'w)\geq 0$. Hence, we can recover the spectral gap inequality \eqref{intro2} from \eqref{eq:chenwand}. 
We conclude with an immediate corollary which will be useful for proving the results given in  the examples.
\begin{corollary}
\label{coro LS}
Assume that $p, w\in C^2(]a,b[)$. If
$f\in C^{\infty}(]a,b[)\cap L^2(p)$ is such that $f'>0$, $-\LS f=\lambda f$ on $]a,b[$ and $\lim_{t \to a}f'(t)w(t)p(t)=0=\lim_{t \to b}f'(t)w(t)p(t)$, then $C(p,w)=\lambda^{-1}$ and $f$ saturates $\PIpw$.

\end{corollary}

\subsection{Recursive approximation of $C(p,w)$}

In this Subsection, we assume that $C(p,w)<\infty$ and $L^2(pw)\subset L^1_{\text{loc}}(]a,b[)$. 
We now propose, using properties of $\Li$, various methods allowing to estimate
$C(p,w)$. We recall that $\{\kappa_i:i\in\N_0\}$ and $\{e_i:i\in\N_0\}$ are respectively the eigenvalues and the eigenvectors of $\Li$. According to Proposition \ref{AAA}, if $\Li$ is compact, we have $\kappa_1=C(p,w)$. Hence, our goal is to find a way of
extracting the first eigenvalue of $\Li$. 
We begin by constructing a sequence of nested intervals containing $C(p,w)$. 
\begin{theorem}
\label{nested}
Assume that (H1)-(H2) hold. For all $g_0\in E^2(p,w)$ such that $g_0>0$, the sequence of intervals 
\begin{equation*}
I_n=\left[\inf \frac{\Li ^{n+1} g_0}{\Li ^n g_0},\sup \frac{\Li ^{n+1} g_0}{\Li ^n g_0}\right]
\end{equation*}
satisfies $I_{n+1}\subset I_n$ for all $n\in\N$ and $C(p,w)\in \bigcap_{n\in\N}I_n$.
\end{theorem}

\begin{proof}[Proof of Theorem \ref{nested}]
Set $g_{n}=\Li^n g_0$ and $M_n=\sup  {\Li g_n}/{g_n}$ for $n\in\N$. For a.e. $x\in ]a,b[$, we compute
\begin{align*}
\frac{\Li g_{n+1}}{g_{n+1}}(x)
&=\frac{1}{g_{n+1}(x)}\int_a^b \frac{K(x,y)}{p(x)w(x)}\Li g_{n}(y)\,dy\\
&=\frac{1}{g_{n+1}(x)}\int_a^b \frac{K(x,y)g_{n}(y)}{p(x)w(x)}\frac{\Li g_{n}(y)}{g_{n}(y)}\,dy\\
&\leq \frac{M_n}{g_{n+1}(x)}\int_a^b \frac{K(x,y)g_{n}(y)}{p(x)w(x)}\,dy\\
&=\frac{M_n}{g_{n+1}(x)}\Li g_{n}(x)=M_n.
\end{align*}
With the same reasoning, we can get $\inf   {\Li g_{n}}/{g_{n}}\leq \inf   {\Li g_{n+1}}/{g_{n+1}}$. Hence, we have $I_{n+1}\subset I_n$. Remark that $g_{n+1}=\Li g_{n}>0$ for all $n\in\N$. So, we can choose $h=Ig_n$ in Theorem \ref{lem1}, which entails that $C(p,w)\in I_n$ for all $n\in\N$ since
\begin{equation*}
-\frac{\Ti h}{h'w}=-\frac{\Ti Ig_n}{wg_n}=\frac{\Li g_n}{g_n}.\qedhere
\end{equation*} 
\end{proof}

As mentioned in the Introduction, there is no guarantee that $\bigcap_{n\in\N}I_n=\left\{C(p,w)\right\}$. To obtain convergence results, we need $\Li$ to be compact, which allows us to use the spectral properties studied in Section \ref{sec:preliminaries}. First, we show that
$I\Li ^n g_0$ is a minimising sequence for $\mathrm{PI}(p,w)$,
irrespective of the starting function $g_0\in
E^2(p,w)$ (see Appendix \ref{sn:furthproo} for a proof). 
\begin{proposition}
\label{suite_mini}
Assume that (H1)-(H2) hold and $\Li$ is compact. For all $g_0\in E^2(p,w)$ such that $\E[g_0e_1w]\neq 0$, we have
$$\lim_{n\rightarrow\infty}\frac{\Var_p\left[I\Li^ng_0\right]}{\E_p\left[\left|\Li^ng_0\right|^2\right]}=C(p,w).$$
\end{proposition}

Proposition \ref{prop:eigen} provides  an easy way of checking that
$\E_p[g_0e_1w]\neq 0$. Indeed, since we have either $e_1> 0$ on
$]a,b[$ or $e_1< 0$ on $]a,b[$, it suffices that $g_0>0$ to ensure
$\E_p[g_0e_1w]\neq 0$.  Proposition \ref{suite_mini} gives us a
theoretical way of estimating $C(p,w)$ but is difficult to use from a
computational point of view. However, we can establish other results
of convergence to
$C(p,w)$. We recall that $\left\{e_i:i\in\N_0\right\}$ is a Hilbert basis of $E^2(p,w)$ by (A2). Thus, any $g_0\in E^2(p,w)$ can be written as $g_0=\sum_{i=1}^{\infty}a_ie_i$ with $a_i=\E_p[g_0e_iw]$. 
We introduce the auxiliary operator $A:={\Li}/{\kappa_1}$, along with its iterations $A^1=A$, $A^{n+1} = A(A^n)$ for $n \in \N_0$. The following holds. 
\begin{theorem}
\label{convergence1}
Assume that (H1)-(H2) hold and  $\Li$ is compact. Let $g_0=\sum_{i=1}^{\infty}a_ie_i\in E^2(p,w)$. The sequence $(A^n g_0)$ 
satisfies
\begin{equation*}
\left\|A^n g_0-a_1e_1\right\|_{L^2(pw)}\leq\left(\frac{\kappa_2}{C(p,w)}\right)^{n} \left\|g_0-a_1e_1\right\|_{L^2(pw)}.
\end{equation*}
In particular, $g_n$ converges to $a_1e_1$ in $L^2(pw)$. 
\end{theorem}

\begin{proof}[Proof of Theorem \ref{convergence1}]
Let $\left\{\mu_i={\kappa_i}/{\kappa_1}:i\in\N_0\right\}$ be the eigenvalues of $A$. 
By (A1) and Proposition \ref{prop:eigen}, we know that $\mu_1=1$, $\mu_i>0$ and $\mu_{i+1}\leq \mu_i<1$ for all $i\geq 2$. 
As $A$ is continuous, we have
\begin{align*}
A^{n}g_0&
=\sum_{i=1}^{\infty}a_i A^n e_i
=\sum_{i=1}^{\infty}\mu_i^n a_i e_i=a_1e_1+\sum_{i=2}^{\infty}\mu_i^na_i e_i.
\end{align*}
So, we can compute using the orthonormality of $\left\{e_i:i\in\N_0\right\}$
\begin{align*}
\left\|A^n g_0-a_1e_1\right\|^2_{L^2(pw)}&=\left\|\sum_{i=2}^{\infty}\mu_i^na_i e_i\right\|^2_{L^2(pw)}\\
&=\sum_{i=2}^{\infty}\mu_i^{2n}a_i^2
\leq \mu_2^{2n}\sum_{i=2}^{\infty}a_i^2\leq\mu_2^{2n} \left\|g_0-a_1e_1\right\|_{L^2(pw)}^2.
\end{align*}
Since $\mu_2={\kappa_2}/{C(p,w)}$, we get the desired inequality.
\end{proof}

\begin{remark}
For any functions $f,g\in L^2(pw)$, the expression
\begin{equation*}
pw(x)\Li f(x)=\int_a^b K(x,\cdot)f\quad\text{and}\quad\frac{\Li f}{\Li g}(x)=\frac{\int_a^b K(x,\cdot)f}{\int_a^b K(x,\cdot)g}
\end{equation*}
makes sense for all $x\in ]a,b[$ such that ${K(x,\cdot)}/{pw}\in L^2(pw)$, despite the fact that $f$, $g$ and $pw$ are not necessarily well defined in $x$. To make the expressions shorter, we will sometimes write $\Li f(x)$ or $e_1(x)$ by abuse of notation, but all our computations make sense if we multiply by $pw(x)$. 
\end{remark}

\begin{theorem}
\label{convergence2}
Assume that (H1)-(H2) hold, $\Li$ is compact and
$k_{w}(x, \cdot) \in L^2(pw)$ for some $x\in ]a,b[$. 
For all $g_0\in E^2(p,w)$ such that
$\E_p[g_0 e_1w]\neq 0$, we have
$$\lim_{n\rightarrow \infty}\left(\frac{\Li ^{n+1}g_0}{\Li ^ng_0}\right)(x)=C(p,w).$$
\end{theorem}
\begin{proof}[Proof of Theorem \ref{convergence2}]
Set $g_n=A^n g_0$ for all $n\in \N$. Assume that $a_1=\E_p[g_0 e_1w]=1$ without loss of generality.  We have
\begin{equation*}
|g_n-e_1|(x)=\left|\frac{\Li g_{n-1}}{\kappa_1}-\frac{\Li e_1}{\kappa_1}\right|(x)\leq\frac{1}{\kappa_1(pw)(x)}\int_a^b K(x,\cdot)|g_{n-1}-e_1|.
\end{equation*}
With the Cauchy-Schwarz inequality and the assumption on $k_{w}$, we obtain
\begin{equation*}
\frac{1}{(pw)(x)}\int_a^b \frac{K(x,\cdot)}{\sqrt{pw}}|g_{n-1}-e_1|\sqrt{pw}\leq \left\|k_{w}(x,\cdot)\right\|_{L^2(pw)}\left\|g_{n-1}-e_1\right\|_{L^2(pw)}.
\end{equation*}
We conclude using Theorem \ref{convergence1} that  $\lim_{n\rightarrow\infty}g_n(x)= e_1(x)$. Now, remark that $e_1(x)=\kappa_1^{-1}\int_a^bk_w(x,\cdot)e_1pw\neq 0$ since $e_1>0$ a.e.\ by Proposition \ref{prop:eigen}. Finally, we just have to observe how 
\begin{equation*}
\frac{1}{C(p,w)}\frac{\Li ^{n+1}g_0}{\Li ^ng_0}(x)=\frac{g_{n+1}}{ g_n}(x)\rightarrow \frac{e_1}{e_1}(x)=1, 
\end{equation*}
as required. 
\end{proof}

We learn from the proof of this proposition that
\begin{equation}
\label{vitesse2}
|A^n g_0-a_1e_1|(x)\leq \frac{\left\|k_{w}(x,\cdot)\right\|_{L^2(pw)}}{C(p,w)}\left\|A^{n-1}g_0-a_1e_1\right\|_{L^2(pw)}.
\end{equation}
If $k_{w}\in L^2(pw\otimes pw)$, the condition $k_{w}(x,\cdot)\in L^2(pw)$ is verified for almost every $x\in ]a,b[$ and \eqref{vitesse2} entails simple convergence of $A^n g_0$ to $a_1e_1$ and of ${\Li ^{n+1}g_0}/{\Li ^ng_0}$ to $C(p,w)$. 
Putting together Theorem \ref{convergence1} and  \eqref{vitesse2} provides an idea of the convergence rate, as follows (see Appendix \ref{sn:furthproo} for a proof).

\begin{proposition}
\label{rate3}
Assume that (H1)-(H2) hold, $\Li$ is compact and $k_{w}(x,\cdot)\in L^2(pw)$ for a certain $x\in ]a,b[$. For all $g_0=\sum_{i=1}^{\infty}a_i e_i\in E^2(p,w)$ with $\E_p[g_0 e_1w]\neq 0$, we have 
\begin{equation*}
\left|\frac{\Li ^{n+1} g_0}{\Li ^n g_0}(x)-C(p,w)\right|\leq 2\left(\frac{|a_1e_1(x)|}{B_n(x)}-\frac{1}{C(p,w)}\right)^{-1}
\end{equation*}
for all $n\in\N$ such that $B_n(x)< C(p,w)|a_1e_1(x)|$, where 
\begin{equation*}
B_n(x):=\left\|k_w(x,\cdot)\right\|_{L^2(pw)}\left\|g_0-a_1e_1\right\|_{L^2(pw)}\left(\frac{\kappa_2}{C(p,w)}\right)^{n-1}.
\end{equation*}
\end{proposition}

Of course, we don't know $C(p,w)$, $\kappa_2$ and $a_1e_1(x)$, so this
convergence rate can't be used in practice. At least, it emphasises
how the speed of convergence depends on the parameter. 
It tells us also that the convergence of ${\Li ^{n+1} g_0}/{\Li ^n g_0}$ to $C(p,w)$ is uniform when $k_{w}\in L^2(pw\otimes pw)$ and $e_1(x)\geq \alpha \left\|k_w(x,\cdot)\right\|_{L^2(pw)}$ for some $\alpha>0$ and almost every $x\in ]a,b[$. In such cases, the sequence of nested intervals $I_n$ from Theorem \ref{nested} converges to $\left\{C(p,w)\right\}$.

\section{Examples}
\label{sec:examples}

In this Section we treat several illustrative examples, some of which serve to connect with the literature (Examples \ref{ex:weightgau}, \ref{exemple uniforme}, and \ref{ex:bonnefj}) the  others containing  new results. All proofs are provided in  Appendix \ref{sec:profex}, while  numerical and symbolic evaluations are available in the  \texttt{Mathematica} file provided in   \cite{GeSwSup}. 

\begin{example}[Bounds from \eqref{eq:stkb} and \cite{NouVi}]
Let $N$ be a centered Gaussian random vector with covariance matrix $K$ and $f:\R^n\to \R$ be a $C^1$ function. Set $Z=f(N)-\E[f(N)]$. Then, we know from \cite{NouVi} that the corresponding Stein kernel is
\begin{equation}
    \tau(x) = \int_0^{\infty} e^{-t} \E \left[ \sum_{i,j=1}^n K_{ij}\partial_i f(N)\partial_j f(e^{-t}N+\sqrt{1-e^{-2t}}N')|Z=x\right] \mathrm{d} t
\end{equation}
where $N'$ is an independent copy of $N$. Assume furthermore that there exist $\alpha_i, \beta_i \ge 0$ with $\alpha_i \le \frac{\partial{f}}{\partial{x_i}}(x) \le \beta_i$  for all $i \in \{1, \ldots, n\}$ and a.e.\ $x \in \mathbb{R}^n$. Then, $Z$ has a density $p$ which satisfies 
\begin{equation*}
    \sum_{i, j=1}^n K_{ij}\alpha_i \alpha_j  \le C(p, 1) \le \sum_{i, j=1}^n K_{ij}\beta_i \beta_j.
\end{equation*}
The other examples from \cite{NouVi} lead to  similarly flavored  conclusions. 
\end{example}

\begin{example}[Weighted Gaussian bounds from \cite{Joulin}]\label{ex:weightgau}
Consider the Gaussian measure with density $p(x)\propto \mathrm{exp}({-{x^2}/{2}})$ on $\R$ and the weight $w(x)=1/(1+bx^2)$ where $b\geq {1}/{2}$. Theorem~2.2 in \cite{Joulin} informs us  that $C(p,w)=4b$. To get the upper bound on the Poincar\'e constant, \cite{Joulin} uses \eqref{intro2} with the function $f'(x)=(1+bx^2)\mathrm{exp}({{x^2}/{4}})$. This choice of function is allowed in \eqref{eq:chenwand} because $(f'wp)(x)=\mathrm{exp}({-{x^2}/{4}})$  is of bounded variation. The two bounds give the same result since $\Phi(f'w)=0$ on $\R$ in this case. We can thus  also get the result with Theorem~\ref{lem1} using the function $h(x)=\T(f'w)(x)=-{x}\,\mathrm{exp}({{x^2}/{4}})/{2}$, which doesn't depend on $b$. Indeed, we can check that $h\in L^1(p)$, $h$ is decreasing and
\begin{equation*}
-\sup_{\R}\frac{\Ti h}{h'w}=\sup_{x\in\R}2\frac{1+bx^2}{1+\frac{x^2}{2}}=4b.
\end{equation*}
\end{example}

\begin{example}[The uniform case]
\label{exemple uniforme}
Consider the uniform measure on $[0,1]$ and the weight $w=1$. This case is very classical, but it allows us to illustrate our methods. Since $\tau(x) = x(1-x)/2$, bound \eqref{eq:stkb} yields $0 \le C(p,1) \le 0.125$. We can also obtain the exact value (and saturating function). Indeed, taking $g_0=1$ as initial function, we obtain the
following sequence $\Li 1(x)=\frac{1}{2}(x-x^2)$,
$\Li^2 1(x)=\frac{1}{24}(x - 2x^3 + x^4)$, $\Li^3 1(x)=\frac{1}{720} \left(3x - 5x^3 + 3x^5 - x^6\right)$ 
and, more generally
\begin{equation}
  \label{tobeproved}
  \Li^n 1(x)=\frac{(-1)^n}{(2n)!} E_{2n}(x)
\end{equation}
where the $E_{2n}$ are the even-indexed Euler polynomials defined by
\begin{equation}
\label{Eulerpolynomial}
E_{2n}(x)=(-1)^n \frac{4(2n)!}{\pi^{2n+1}}\sum_{k=0}^{\infty}\frac{\sin((2k+1)\pi x)}{(2k+1)^{2n+1}}
\end{equation}
for all $n\geq 1$ (see \cite{web:Euler}). With \eqref{Eulerpolynomial}, we can see that 
$$
\lim_{n\rightarrow\infty}\pi^{2n}\Li^n 1(x)=\frac{4}{\pi}\sin(\pi x)
$$
for all $x\in [0,1]$. 
This is exactly what we expected since $h(x)=\cos(\pi x)$ saturates
$\mathrm{PI}(p,1)$. Using Theorem \ref{nested}, we have after six iterations that
$ C(p,1) \in [0.101319,0.101322]$ while
$C(p,1)={\pi^{-2}}\approx 0.10132$. Figure~\ref{1} illustrates the
first three ratios $({\Li^{n}1}/{\Li^{n-1} 1})(x)$ over $x \in ]0,1[$, $n=1, 2, 3$.
\begin{figure}[]
  \includegraphics[width=6cm]{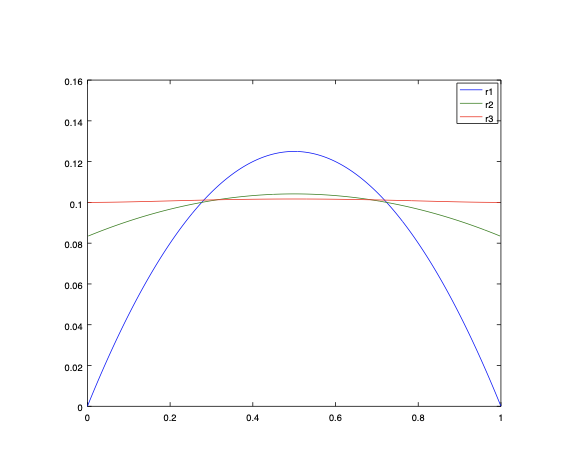}
	\caption{Convergence for the uniform measure on $[0,1]$. rn denotes the ratio ${\Li^{n} 1}/{\Li^{n-1}1}$.}
	\label{1}
\end{figure}
 
\end{example}
 
\begin{example}
\label{beta exact}
Consider the beta distribution with parameters $\alpha>0$ and $\beta>0$, whose density is $p(x) \propto x^{\alpha-1}(1-x)^{\beta-1}$ on $[0,1]$.  
When $\beta=1$, as in the previous example, we can  obtain the exact value (and saturating function) with our results. Define $g(x)=x^{1-\frac{\alpha}{2}}J_{\frac{\alpha}{2}}(2x)$ where $J_{\frac{\alpha}{2}}$ denotes a Bessel function of the first kind. The Poincar\'e constant of $p$ with weight $w = 1$ is  
$$
C(p,1)=\frac{1}{4r_1^2}
$$
where $r_1$ is the smallest positive root of $g$. The saturating function is 
$$
h(x)=x^{1-\frac{\alpha}{2}}J_{\frac{\alpha}{2}-1}(2r_1x).
$$
Moreover, if $r$ is a positive root of g, then ${1}/({4r^2})$ is an eigenvalue of $\Li$ associated with the eigenvector 
$
e(x)=g(rx).
$ 
The case $\alpha = 1, \beta>0$ can be treated similarly. 
We do not have an explicit formula of $C(p,1)$ for other combinations  of $\alpha, \beta$. Nevertheless, since $\tau(x) = x(1-x)/(\alpha + \beta)$, \eqref{eq:stkb} can be applied  yielding
\begin{equation}\label{eq:boubet}
 0 \le C(p, 1) \le \frac{1}{4(\alpha + \beta)}.
\end{equation}
Pushing the arguments to the second order through  Theorem \ref{nested}, some (tedious) computations lead to 
\begin{equation*}
    \frac{\Li^{2} 1}{\Li1}(x) = - \frac{x^2}{3(2+\alpha+\beta)} + \frac{(2+\alpha+3\beta)x}{6(1+\alpha+\beta)(2+\alpha+\beta)} + \frac{(1+\alpha)(2+\alpha+3\beta)}{6(\alpha+\beta)(1+\alpha+\beta)(2+\alpha+\beta)}
\end{equation*}
from which it follows that 
\begin{align}\label{eq:sforderbeta}
    &\min(1+\alpha, 1+\beta) \frac{2+\alpha+3\beta}{6(\alpha+\beta)(1+\alpha+\beta)(2+\alpha+\beta)}  \le C(p, 1) \nonumber \\ & \le  \frac{(2+3\alpha+\beta)(2+\alpha+3\beta)(4+3\alpha+3\beta)}{48(\alpha+\beta)(1+\alpha+\beta)^2(2+\alpha+\beta)}.
\end{align}
We could obviously iterate further. 
We can also use Theorems \ref{nested} and \ref{convergence2} to numerically approximate $C(p,1)$ for specific values of $\alpha$ and $\beta$.  We report some  values  in Table \ref{tab:2} along with the bounds obtained from \eqref{eq:boubet}, \eqref{eq:sforderbeta},  Proposition \ref{bound_k} and the approximation obtained through the \texttt{R} code from \cite{Roustant}. The 7th column reports the approximation obtained after 8 iterations. 
\begin{table}[!]
    \centering
\begin{tabular}{|c|c|c|c|c|c|c|c|}
\hline
$\alpha$ & $\beta$ & $\left\|k\right\|$ & \eqref{eq:boubet} & \eqref{eq:sforderbeta} & $I_4$ & C(p,1)& \cite{Roustant} \\
\hline
2 & 2 & 0.0579 & 0.062 & [0.04166, 0.05555] & [0.05390, 0.054012] & 0.05408 & 0.05408 \\
0.5 & 3 & 0.0557 & 0.071 &[0.03318, 0.05792] & [0.05051, 0.05286] &$0.0528$  & 0.06897 \\
3 & 2 & 0.0471 & 0.050 &[0.03095, 0.04492] & [0.04294, 0.04358]& $0.04341$ & 0.04334 \\
\hline
\end{tabular}
\caption{Numerical data for Example \ref{beta exact}}
\label{tab:2}
\end{table}

\end{example}

\begin{example}
\label{gamma}
Consider the gamma measure with density $p(x)\propto x^{k-1}e^{-{x}/{\theta}}$ for $x\in\R^+$ and $k,\theta\in ]0,\infty[$. Since the Stein kernel of the gamma is linear, bound \eqref{eq:stkb} is not informative. We can  obtain the exact value and saturating functions when $k>1$. In fact, the Poincar\'e constant of $p$ with weight $w=1$ is
$$
C(p,1)=\frac{(k+1)^2}{k}\theta^2
$$
and the saturating function is
$$
h(x)=\left(x-\theta(k+1)\right) \exp\left(\frac{x}{\theta(k+1)}\right).
$$
When $k\in ]0,1]$, integrability issues make the lower bound unusable, so the best we  have is  $C(p,1)\leq {(k+1)^2}\theta^2/{k}$. It can be seen that this upper bound is then worse than the bound obtained through \eqref{eq:roustbobokov} 
 (whereas \eqref{eq:ourroust} yields a trivial upper bound).    
\end{example}

\begin{example}
  Consider the Subbotin measure with density
  $p(x) \propto \mathrm{exp}({-{|x|^{\alpha}}/{\alpha}})$. The Stein kernel is 
  \begin{equation*}  \tau_\alpha(x) = e^{|x|^\alpha/\alpha} \alpha^{2/\alpha - 1} \Gamma(2/\alpha, |x|^\alpha/\alpha)
  \end{equation*}
  ($\Gamma(a, x) = \int_x^{\infty} t^{a-1} e^{-t}dt$ is the incomplete gamma function). 
  One can easily show that (i)  if $\alpha <2$ then  $\tau_\alpha(x)$ is strict concave with minimal value 0 and unbounded from above, (ii) $\tau_2(x) =1 $ (as expected), and (iii) if $\alpha>2$ then $\tau_\alpha(x)$ is  strict convex, with maximal value $\tau_\alpha(0) = \alpha^{2/\alpha-1} \Gamma(2/\alpha)$ and minimal value 0. It follows from \eqref{eq:stkb} that 
$ 0 \le C(p, 1) \le \tau_\alpha(0)$ 
  for all $\alpha>2$. One can see that our upper bound provides a mild improvement over the upper bound from 
  \cite{Joulin} in this case. Combining with the lower bound from that same paper, we deduce that 
\begin{equation}
\label{bound_subbotin}
3^{\frac{2}{\alpha}-1}\leq C(p,1)\leq \alpha^{\frac{2}{\alpha}-1} \Gamma(\frac{2}{\alpha}) 
\end{equation}
for all $\alpha>2$.  
Looking now at specific values of the parameters, if we take  $\alpha=3$ and weight $w=1$ then \eqref{bound_subbotin} yields $0.693\le C(p, 1) \le 0.939$. Proposition \ref{bound_k} enhances the previous upper bound because
$C(p,1)\leq \left\|k\right\|_{L^2(p\otimes p)}\approx 0.89442$. Furthermore, applying Theorem \ref{convergence2} with  $g_0=1$, the first elements of the sequence
$\big( {\Li ^{n+1}1}/{\Li ^n 1}\big)(0)$ are approximately
\begin{equation*}
0.93889,\ 0.82934,\ 0.81074,\, 0.80739,\, 0.80858,\, \ldots
\end{equation*}
leading to $C(p,1)\approx 0.8$, which is consistent with \eqref{bound_subbotin}. Other values of the parameters can be treated similarly.  
\end{example}

\begin{example}
\label{weibull}
Consider the Weibull density $p(x)\propto x^{k-1}\exp(-\left({x}/{\lambda}\right)^k)$ with parameters $k,\lambda>0$ and the weight $w=x^{2-k}$. The weighted Poincar\'e constant is $C(p,w)=\lambda^k/k^2$ and the saturating function is $h(x)=x^k-\lambda^k$. Let $L^{(\alpha)}_i$ be the Laguerre polynomial of degree $i$ with parameter $\alpha$. We guess that $e_i(x)=x^{k-1}L^{(1)}_{i-1}\left(x^k/\lambda^k\right)$ is an eigenvector of $\Li$ associated to the eigenvalue $\lambda^k/(k^2i)$ for all $i\in\N$, but we are not able to provide a proof of this statement through our methods. 

Regarding the weight $w=1$, we are not able to obtain exact results (except when $k=2$). We can use \eqref{eq:stkb} because $\tau(x)$ has an explicit expression for this density providing non trivial upper bounds as soon as $k>1, \lambda>2$, namely 
\begin{equation*}
    \tau_{k, \lambda}(x) = \frac{\lambda^k}{k^2} x^{1-k} \left( k x - \lambda \Gamma(1/k) + e^{(x/\lambda)^k} \lambda \Gamma(1/k, (x/\lambda)^k) \right).
\end{equation*}
We can also use numerical approximations along the lines of the previous examples and urge the interested reader to consult  the supplementary material \cite{GeSwSup} for illustrations and pretty plots. 
\end{example}

\begin{example}[Weighted Gaussian bounds from \cite{Joulin}]\label{ex:bonnefj}
Consider as in Example \ref{ex:weightgau} the standard Gaussian measure with density $p(x)\propto \mathrm{exp}({-{x^2}/{2}})$ and the weight $w(x)={1}/{(1+bx^2)}$ where $b\geq 0$. Thanks to \cite{Joulin}, $C(p,w)$ is known and given by
\begin{equation*}
C(p,w)=\begin{cases}
\frac{1}{1-b}&\text{if  }0\leq b\leq \frac{1}{2},\\
4b&\text{if  }b\geq \frac{1}{2}.
\end{cases}
\end{equation*}
Remark that 
\begin{equation*}
\left\|k_{w}\right\|^2_{L^2(pw\otimes pw)}=\int_{\R}\int_{\R}\frac{K(x,y)^2}{p(x)p(y)}\left(1+bx^2\right)\left(1+by^2\right)dx\,dy.
\end{equation*}
Hence, as $b$ increases, the bound given by Proposition \ref{bound_k} increases. This implies that the convergence of the sequence of ratios  $\big({\Li ^{n+1}g_0}/{\Li ^ng_0}\big)(x)$ towards $C(p,w)$ gets slower, as shown by Theorem \ref{rate3}. It therefore seems that our numerical approach is better suited to small values of $b$. Fix, for the sake of illustration,   $b=0.1$. Then 
$\left\|k_w\right\|_{L^2(pw\otimes pw)}\approx 4.266$ 
and, with $g_0=1$ and $x=0$, the first elements of $\big({\Li ^{n+1}1}/{\Li ^n 1}\big)(x)$ are approximately
\begin{equation*}
1.06667,\ 1.0925,\ 1.10507,\ \ldots
\end{equation*}
This sequences converges to $C(p,w)=(1-b)^{-1}\approx 1.11111$, as predicted by \cite{Joulin}.
\end{example}

\appendix

\section{Relation between $\Li$ and $\LS$}
\label{sec:relll}
In this section, we show that $\Li$ has, in some sense, the inverse spectral properties of $\LS$. This will be helpful in the proof of  Proposition \ref{C12}. We assume that $C(p,w)<\infty$ and $L^2(pw)\subset L^1_{\text{loc}}(]a,b[)$.
The Sturm-Liouville operator $\LS$ has already been defined on the space of twice
differentiable functions in the introduction. Now, we give a weak version
valid on the whole space $H^1(p,w)$. We define
$\LS:H^1(p,w)\rightarrow H^1_c(p,w)^*$ by
\begin{equation*}
\LS h:H^1_c(p,w)\rightarrow \R:v\mapsto -\E_p\left[h'v'w\right]
\end{equation*}
where $H^1_c(p,w)^*$ is the dual space of $H^1_c(p,w)$. We can define its pseudo-inverse operator $\LS^{-1}:L^2(p)\rightarrow H^1_c(p,w)$ by associating to $f\in L^2(p)$ the solution $h\in H^1_c(p,w)$ of
\begin{equation}
\label{SL faible}
-\E_p\left[h'v'w\right]=\E_p\left[fv\right], \quad \forall v\in H^1_c(p,w). 
\end{equation}
This solution exists and is unique by the Riesz representation
theorem. Indeed, using Cauchy-Schwarz and Poincar\'e inequalities, we
can see that the form $v\in H^1_c(p,w)\mapsto \E_p\left[fv\right]$ is
in $H^1_c(p,w)^*$. If $h\in H^1(p,w)$, we have by Theorem \ref{theo2} 
\begin{equation*}
\E_p\left[\Li h'\,v'w\right]=\E_p\left[-\Ti h\,v'\right]=\E_p\left[hv\right]=-\E_p\left[(\LS^{-1}h)'v'w\right]
\end{equation*}
for all $v\in H^1_c(p,w)$. By uniqueness of the solution of \eqref{SL
  faible}, it follows that $-I\Li h'=\LS^{-1}h$. The operators
$\LS^{-1}$ and $\Li$ are thus related by
the following identity:
\begin{equation}\label{LS-1}
-\LS^{-1}=
I\Li D
\end{equation}
on $ H^1(p,w)$.  We say that $\lambda\in\mathbb{C}$ is an eigenvalue
of $-\LS$ if there exists  $h\in H^1(p,w)$ such that
\begin{equation}
\label{eigen LS}
\E_p\left[h'v'w\right]=\lambda\E_p\left[hv\right], \quad \forall v\in H^1_c(p,w). 
\end{equation}
First, observe that $h\in H^1_c(p,w)$ is an eigenvector of $-\LS$ with eigenvalue $\lambda$ if and only if $h$ is an eigenvector of $-\LS^{-1}$ with eigenvalue $\lambda^{-1}$. Then, remark that
\begin{align*}
\Li h'=D(-\LS^{-1})Ih'=D(-\LS^{-1})h=D(\frac{1}{\lambda} h)=\frac{1}{\lambda} h'.
\end{align*}
This leads us to the following.
\begin{corollary}
\label{relation vp}
Assume that (H1)-(H2) hold. A function $h\in H^1_c(p,w)$ is an eigenvector of $-\LS $ with eigenvalue $\lambda$ if and only if $h'$ is an eigenvector of $\Li$ with eigenvalue $\lambda^{-1}$. 
\end{corollary}

We end this section by recalling a classical result.
\begin{proposition}
\label{eigen}
Assume that (H1)-(H2) hold. If $H^1(p,w)$ is dense in $L^2(p)$ and the injection of $H^1(p,w)$ in $L^2(p)$ is compact, the eigenvalues of $-\LS$ form an increasing sequence $\left\{\lambda_i:i\in\N\right\}$ of real positive numbers converging to infinity and such that $\lambda_1=C(p,w)^{-1}$. Moreover, there exists a Hilbert basis of $H^1(p,w)$ of associated eigenvectors $\left\{v_i:i\in\N\right\}$. 
\end{proposition}
\begin{proof}[Proof of Proposition \ref{eigen}]
The symmetric bilinear form $a(h,v):=\E_p\left[h'v'w\right]$ is continuous and coercive on $H^1_c(p,w)$. The hypotheses (C1)-(C2) ensure that $H^1_c(p,w)$ is dense in $L^2_c(p)$ and that the injection of $H^1_c(p,w)$ in $L^2_c(p)$ is compact. So, we can apply theorem 7.3.2 in \cite{Allaire} with $H=L^2_c(p)$ and $V=H^1_c(p,w)$. It tells us exactly our statement except the equality $\lambda_1=C(p,w)^{-1}$ which follows from proposition 7.3.4 in \cite{Allaire}.

\end{proof}

\section{Further proofs}
\label{sn:furthproo}

\begin{proof}[Proof of curious identity]
Let   $h_2(x) = \int_c^x{p}/{(P \bar{P})}$ for some $c\in ]a,b[$. First, we compute using the Fubini-Tonelli Theorem 
\begin{align*}
    \E_p[|h_2|]&=\int_a^c p(x)\int_x^c\frac{p(y)}{P(y)\bar{P}(y)}dy\,dx+\int_c^b p(x)\int_c^x\frac{p(y)}{P(y)\bar{P}(y)}dy\,dx\\
    &=\int_a^c \frac{p(y)}{P(y)\bar{P}(y)}\int_a^y p(x)dx\,dy+\int_c^b \frac{p(y)}{P(y)\bar{P}(y)}\int_y^b p(x)dx\,dy\\
    &=\int_a^c \frac{p(y)}{\bar{P}(y)}dy+\int_c^b \frac{p(y)}{P(y)}dy\\
    &=-\int_a^c \log (\bar{P}(y))'dy+\int_c^b \log(P(y))'dy\\
    &=-\log (\bar{P}(c))-\ln(P(c))=-\log(P(c)\bar{P}(c)).
\end{align*}
As the last expression is finite, we have $h_2\in L^1(p)$. Since $h_2$ is increasing, we can apply the second identity in \eqref{eq:saumaaaa} to obtain 
 \begin{align*}
      \mathrm{Var}_p\left[h_2\right] &    = \int \int  \frac{p(x)}{P(x) \bar{P}(x)} P(x \wedge  y) \bar{P}(x \lor y) \frac{p(y)}{P(y) \bar{P}(y)} dx dy\\
     & = \int_a^b \int_a^x  \frac{p(x)}{P(x) }   \frac{p(y)}{ \bar{P}(y)}  dy dx + \int_a^b \int_x^b\frac{p(x)}{ \bar{P}(x)}  \frac{p(y)}{P(y) } dy dx\\
     & =- \int_a^b \log(P(x))' \log (\bar P(x)) dx + \int_a^b \log(P(x)) \log (\bar P(x))' dx\\
     & = - \left[ \log(P(x)) \log (\bar P(x))\right]_a^b + 2 \int_a^b \log(P(x)) \log (\bar P(x))' dx\\
     & = 0 - 2 \int_0^1 \frac{\log(u)}{1-u}  du = \frac{\pi^2}{3}.
 \end{align*}
 Hence
\begin{equation*}
    \mathrm{Var}_p\left[h_2\right] = \frac{\pi^2}{3}
\end{equation*}
 for all densities $p$ on the real line. 
\end{proof}
\begin{proof}[Proof of Lemma \ref{dif_increasing}]
Let $f\in H^1(p,w)$. Remark that $f'^+$ and $f'^-$ still belong to $L^2(pw)\cap L^1_{\text{loc}}(]a,b[)$ and so to $E^2(p,w)$ by Proposition \ref{ID}. Hence, $f_1:=I(f'^+)$ and $f_2:=I(f'^-)$ belong to $H^1(p,w)$ by Proposition \ref{ID}. Moreover, we have $f_1-f_2=I[f'^+-f'^-]=I[f']=f$ and $f_1$ and $f_2$ are increasing.
\end{proof}

\begin{proof}[Proof of Lemma \ref{lem increasing}]
Let $h\in H^1_c(p,w)$ be non monotone. Take $c\in ]a,b[$ and set $g=\int_c^{\cdot}|h'|$. We have obviously $\E_p[|h'|^2]=\E_p[|g'|^2]$ and $g$ is increasing. First, assume that $g\in L^2(p)$. 
As $g$ is continuous, we can find  $d\in]a,b[$ such that $g(d)=\E_p[g]$. Then 
\begin{equation}
\label{avar}
\Var_p[h]\leq \E_p\left[|h-h(d)|^2\right]=\E_p\left[\left|\int_d^{\cdot} h'\right|^2\right]\leq\E_p\left[\left(\int_d^{\cdot}|h'|\right)^2\right]=\Var_p[g].
\end{equation}
The second inequality is an equality if and only if $h'$ doesn't change sign on $]a,d]$ and $[d,b[$. If $h'$ is positive on $]a,d[$ and negative on $]d,b[$ (or vice versa), it means that $h(d)$ is an extremum of $h$. So, the first inequality in \eqref{avar} would be strict since $h(d)\neq \E_{p}[h]$. Therefore, \eqref{avar} is an equality if and only if if $h$ is monotone on $]a,b[$.

Now assume that $g\notin L^2(p)$. Set $g_n(x)=\max\left\{-n,\min\left\{g,n\right\}\right\}$. We have $g_n\in H^1(p,w)$, $\E_p[|g_n'|^2w]\leq \E_p[|h'|^2w]$ and $g_n$ is increasing for all $n\in\N$. They are two possibilities : either $(\E_p[g_n])$ is bounded or there exists a subsequence of $(g_n)$, still written $(g_n)$, such that $\lim_{n\to\infty}\E_p[g_n]\in\left\{-\infty,\infty\right\}$. In the first case, we have by monotone convergence
$$
\lim_{n\to\infty}\Var_p[g_n]=\lim_{n\to\infty}\left(\E_p[g_n^2]-\E_p[g_n]^2\right)=\infty.
$$
In the second case, we have by the Fatou's lemma
\begin{equation*}
{\lim\inf}_{n\to\infty}\Var_p[g_n]\geq \int_a^b{\lim\inf}_{n\to\infty}\left(g_n-\E_p[g_n]\right)^2p.
\end{equation*}
Since $\lim_{n\to\infty}\left(g_n-\E_p[g_n]\right)^2p=\infty$ a.e. we get also $\lim_{n\to\infty}\Var_p[g_n]=\infty$.
So, in both cases, we can find a $n\in\N$ such that $\Var_p[g_n]> \Var_p[h]$ and we get the desired result. 
\end{proof}

\begin{proof}[Proof of Theorem \ref{theo2}]
Take $g, h\in L^1(p)$, weakly differentiable and increasing. There exists $c\in ]a,b[$ such that $h\leq \E_p[h]$ on $]a,c]$. As the function $(x,y)\mapsto g'(x) 1_{]a,x]}(y)(\E_p[h]-h(y))p(y)$ is measurable and positive on $]a,c[\times ]a,c[$, we can use the Fubini-Tonelli Theorem to obtain 
\begin{align*}
\int_a^c -g'(x)\Ti h(x) p(x)dx&=-\int_a^c g'(x)\int_a^c 1_{]a,x]}(y)(h(y)-\E_p[h])p(y)dy\,dx\\
&=-\int_a^c (h(y)-\E_p[h])p(y)\int_a^c 1_{[y,b[}(x)g'(x)dx\,dy\\
&=-\int_a^c (h(y)-\E_p[h])p(y)\int_y^c g'(x)dx\,dy\\
&=\int_a^c (h(y)-\E_p[h])\left(g(y)-g(c)\right)p(y)dy.
\end{align*}
Using the equivalent representation $p\Ti h(x)=\int_x^b\left(\E_p[h]-h\right)p$ (see Definition 2.5 in \cite{Swan}), we can show that 
\begin{equation*}
\int_c^b -g'\Ti h p=\int_c^b (h(y)-\E_p[h])\left(g(y)-g(c)\right)p(y)dy
\end{equation*}
by similar computations. Since $-\Ti h\,g', (h-\E_p[h])(g-g(c))\geq 0$, we can put both calculations together to obtain 
\begin{align*}
\E_p\left[-\Ti h\,g'\right]&=\int_a^c -g'\Ti h \,p+\int_a^c -g'\Ti h \,p\\
&=\int_a^b (h(y)-\E_p[h])\left(g(y)-g(c)\right)p(y)dy\\
&=\E_p\left[(h-\E_p[h])g\right]=\Cov_p[g,h].
\end{align*}
The second equality in \eqref{eq:saumaaaa} follows from \eqref{caracL} and the Fubini-Tonelli Theorem since $g',h',k_1\geq 0$. Now, take $g$ and $h$ such that we can write $g=g_1-g_2$ and $h=h_1-h_2$ where $g_1,\,g_2,\,h_1,\, h_2\in L^2(p)$ are increasing functions. Since $\Ti$ is linear, we easily see that \eqref{eq:saumaaaa} is still valid. Finally, remark that $g$ can be written in such form if $g\in H^1(p,w_1)$ for some weights $w_1$
such that $C(p,w_1)<\infty$ by Lemma \ref{dif_increasing}. As the same holds for $h$, the second claim of the Theorem follows.
\end{proof}

\begin{proof}[Proof of Theorem \ref{theo1}]
Take $g\in H^1(p,w_h)$ such that $g'\geq 0$ a.e. We begin by
  showing that the function $G:]a,b[^2\rightarrow\R$ defined by
\begin{equation*}
G(x,y)=\frac{g'(x)}{\sqrt{-h'(x)}}\sqrt{-k_1(x,y)h'(y)}
\end{equation*}
is in $L^2(p\otimes p)$. As $G^2$ is measurable and positive, we have by the Fubini-Tonelli Theorem
\begin{align*}
\int_{]a,b[^2}G^2(x,y)p(x)p(y)dx\,dy&=\int_a^b\left(\int_a^b G^2(x,y)p(y)dy\right)p(x)dx\\
&=\int_a^b\frac{|g'(x)|^2}{h'(x)}\left(\int_a^b \frac{K(x,y)}{p(x)p(y)}h'(y)p(y)dy\right)p(x)dx\\
&=\int_a^b |g'(x)|^2\frac{-\Ti h(x)}{h'(x)}p(x)dx
\end{align*}
where we used \eqref{caracL} for the last equality. The last expression is finite since $g\in H^1(p,w_h)$. Let $X,Y\sim p$. Using Theorem \ref{theo2}, the Cauchy-Schwarz inequality and the previous computation, we get
\begin{align*}
\Var[g(X)]&=\E\left[k_1(X,Y)g'(X)g'(Y)\right]\\
&=\E\left[\frac{g'(X)}{\sqrt{-h'(X)}}\left(-k_1(X,Y)h'(Y)\right)^{\frac{1}{2}}\frac{g'(Y)}{\sqrt{-h'(Y)}}\left(-k_1(Y,X)h'(X)\right)^{\frac{1}{2}}\right]\\
&=\E\left[G(X,Y)G(Y,X)\right]\\
&\leq\E\left[G^2(X,Y)\right]^{\frac{1}{2}}\E\left[G^2(Y,X)\right]^{\frac{1}{2}}\\
&=\E\left[|g'(X)|^2\frac{-\Ti h(X)}{h'(X)}\right]^{\frac{1}{2}}\E\left[|g'(Y)|^2\frac{-\Ti h(Y)}{h'(Y)}\right]^{\frac{1}{2}}\\
&=\E\left[|g'(X)|^2\frac{-\Ti h(X)}{h'(X)}\right].
\end{align*}
Hence $C(p,w_h)\leq 1$ since it is enough to consider increasing functions by Lemma \ref{lem increasing}. Furthermore, this inequality is an equality if and only if $G$ is symmetric. This is true if and only if ${g'}/{h'}$ is constant over $]a,b[$ or, equivalently, $g=\alpha h+\beta$ for some $\alpha,\,\beta\in\R$. This choice of $g$ is allowed as soon as $h\in L^2(p)$. Indeed, by Theorem \ref{theo2}, we have $\E_p[|h'|^2w_h]=-\E_p[h'\Ti h]=\Var_p[h]<\infty$. In particular, $C(p,w_h)=1$ if $h\in L^2(p)$. 
\end{proof}

\begin{proof}[Proof of Corollary \ref{lem:optimstek}]
Let $w$ be a weight such that $C(p,w)=1$. If $\E[w(X)]=\infty$, we have nothing to show. Assume $\E[w(X)]<\infty$. By hypothesis, we have $\text{id}\in L^2(p)$. Remark that $\E\left[w(X)|\text{id}'(X)|^2\right]=\E[w(X)]$. So, we have $\text{id}\in H^1(p,w)$. By the Poincar\'e inequality, we get 
$\Var[\text{id}(X)]\leq \E[w(X)].$
Finally, Theorem \ref{theo2} ensures that $\Var[X]=\E[\tau(X)]$ since $\id\in L^2(p)$ is increasing.
\end{proof}

\begin{proof}[Proof of Proposition \ref{ID}]
The following connections hold between $H^1(p,w)$ and $E^2(p,w)$
\begin{equation}
\label{EH}
f\in E^2(p,w)\Rightarrow \int_c^{\cdot} f\in H^1(p,w) \quad\text{    and    }\quad h\in H^1(p,w)\Rightarrow h' \in E^2(p,w).
\end{equation}
Indeed, if $f\in E^2(p,w)$, we have $f\in L^1_{\text{loc}}(]a,b[)$ and so $\int_c^{\cdot} f\in H^1(p,w)$ by Lemma VIII.2 in \cite{Brezis}. If $h\in H^1(p,w)$, there exists a continuous version of $h$ such that $\int_c^{\cdot} h'=h-h(c)\in L^2(p)$ by Theorem VIII.2 in \cite{Brezis}. 
We deduce that the operators $I:E^2(p,w)\rightarrow H^1_c(p,w)$ and $D:H^1_c(p,w)\rightarrow E^2(p,w)$ are well defined. The statement $I=D^{-1}$ is obvious. The continuity of $I$ and $D$ follows from the choice of norms on $E^2(p,w)$ and $H^1_c(p,w)$.

In order to show that $L^2(pw)\cap L^1_{\text{loc}}(]a,b[)= E^2(p,w)$, we just have to show $L^2(pw)\cap L^1_{\text{loc}}(]a,b[)\subset E^2(p,w)$ as the other inclusion is evident. By Lemma VIII.2 in \cite{Brezis}, every function in $L^1_{\text{loc}}(]a,b[)$ can be written as the weak derivative of a function $h\in L^1_{\text{loc}}(]a,b[)$. Assume by contradiction that there exists a function $h$ weakly differentiable such that $\E_p[h]=0$, $\left\|h'\right\|_{L^2(pw)}=1$ but $\left\|h\right\|_{L^2(p)}=\infty$. For each $n\in\N$, set $h_n(x)=\max\left\{-n,\min\left\{h,n\right\}\right\}$.
We have $\left\|h_n\right\|_{L^2(p)}\leq n$ and $\left\|h_n'\right\|_{L^2(pw)}\leq\left\|h'\right\|_{L^2(pw)}= 1$, so that $h_n\in H^1(p,w)$. As $(h_n^2)$ is an increasing sequence of functions and $h_n\rightarrow h$ a.e.\ we have by monotone convergence that $\left\|h_n\right\|_{L^2(p)}\rightarrow \left\|h\right\|_{L^2(p)}=\infty$. This is a contradiction with the fact that $C(p,w)<\infty$.

Finally, we deal with the four equivalences. The two first statements are equivalent because $E^2(p,w)$ and
  $H^1_c(p,w)$ are homeomorphic by Proposition \ref{ID}. The third
  assertion entails the fourth one since $L^2(pw)\cap L^1_{\text{loc}}(]a,b[)= E^2(p,w)$.  
  The fourth one implies
  the first one because $L^2(pw)$ is a Hilbert space. Indeed, it is
  the case as soon as $pw$ is the density of a $\sigma$-finite measure
  (see section 3.2 in \cite{Heinonen}), which is true since
  $pw\in L^1_{\text{loc}}(]a,b[)$. It remains to show that the first
  one entails the third one. By contradiction, assume that there
  exists $f\in L^2(pw)\setminus L^1_{\text{loc}}(]a,b[)$. Define $f_n(x)=\max\left\{-n,\min\left\{f,n\right\}\right\}$.
We have $(f_n)\subset L^2(pw)\cap L^1_{\text{loc}}(]a,b[)=E^2(p,w)$ and $f_n\rightarrow f$ in $L^2(pw)$ by dominated convergence. Hence, $E^2(p,w)$ isn't closed in $L^2(pw)$. Since $E^2(p,w)$ is a subspace of $L^2(pw)$ endowed with the same norm, this implies that it is not a Hilbert space. \qedhere
\end{proof}

\begin{proof}[Proof of Proposition \ref{Li continuous}]
Let $f\in E^2(p,w)$. For every $h\in E^2(p,w)$, since $C(p,w)<\infty$, we can use Theorem \ref{theo1} to obtain
\begin{equation}
\label{sym}
\E_{p}[h\Li  f\,w]=-\E_p\left[h\Ti If\right]
=\E_p\left[Ih\,If\right].
\end{equation}
Moreover, we have
\begin{equation*}
\E_p\left[Ih\,If\right]\leq \left\|Ih\right\|_{L^2(p)}\left\|If\right\|_{L^2(p)} \leq C(p,w) \left\|f\right\|_{L^2(pw)}\left\|h\right\|_{L^2(pw)}.
\end{equation*}
Putting these two computations together, we see that the form $h\in E^2(p,w)\mapsto \E_{p}[h\Li  f\,w]$ is linear and continuous. By the Riesz representation theorem, there exists a $g\in E^2(p,w)$ such that $\E_{p}[h\Li  f\,w]=\E_{p}[h g\,w]$ for all $h\in E^2(p,w)$. As $pw\in L^1_{\text{loc}}(]a,b[)$, $E^2(p,w)$ contains the indicator functions of  compact sets.  
So, we have
\begin{equation*}
\int_E (\Li  f-g)pw=0
\end{equation*}
for all compact set $E\subset ]a,b[$. This implies that $\Li f=g$ a.e. We conclude that $\Li f\in E^2(p,w)$ and $\Li$ is well defined. Equation \eqref{sym} also shows that $\Li$ is self-adjoint and  positive, in the sense that $\E_{p}[f\Li  f\,w]\geq 0$ for all $f\in E^2(p,w)$.
The Hellinger-Toeplitz Theorem says that a self-adjoint operator defined everywhere on a space is continuous on this space (see the Corollary of Theorem III.12 in \cite{Modern}). Hence, $\Li$ is continuous. Another way of seeing that $\Li$ is continuous is to take $h=\Li f$ in \eqref{sym}. Finally, as $\Li $ is self-adjoint, we have by Proposition 2.13 in Chapter 2 of \cite{Conway}
\begin{equation*}
\left\|\Li \right\|_{E^2(p,w)\rightarrow E^2(p,w)}=\sup_{f\in E^2(p,w),\left\|f\right\|=1}\E_{p}\left[f\Li fw\right]=\sup_{f\in E^2(p,w),\left\|f\right\|=1}\Var_p\left[If\right].
\end{equation*}
Moreover, using that $I$ is a bijection, we can see that
\begin{equation*}
\sup_{f\in E^2(p,w),\left\|f\right\|_{E^2}=1}\Var_p\left[If\right]=\sup_{h\in H^1_c(p,w),\left\|h'\right\|_{E^2}=1}\Var_p\left[h\right]=C(p,w).
\end{equation*}
With these two computations, we get the desired conclusion.
\end{proof}

\begin{proof}[Proof of Proposition \ref{prop:eigen}]
Let $e\in E^2(p,w)$ be an eigenvector of $\Li$ with eigenvalue $\kappa$. As a preliminary remark, observe that 
\begin{equation}
\label{klop1}
\E_p\left[If\, Ie\right]=-\E_p\left[f\Ti Ie\right]=\E_p\left[f\Li e\,w\right]=\kappa\E_p\left[few\right]
\end{equation}
for all $f\in E^2(p,w)$ by Theorem \ref{theo2}. Assume that $e>0$ a.e. Then, we have
$$
\kappa=\frac{\Li e}{e}=\frac{-\Ti Ie}{(Ie)'w}.
$$
Since $Ie\in L^2(p)$, this implies that $\kappa=C(p,w)$ by Theorem \ref{lem1}. 

Now, assume that $\kappa=C(p,w)$. With \eqref{klop1}, we get $\left\|Ie\right\|^2_{L^2(p)}=C(p,w)\E_p\left[e^2w\right]$,
which means that $Ie$ saturates $\PIpw$. By Lemma \ref{lem increasing}, we know that the function which saturates $\PIpw$ must be increasing. Hence, we have $e\geq 0$ a.e. We have still to show that $e>0$ a.e. Assume by contradiction that $e=0$ on a non negligible subset of $]a,b[$. Then, we would have $\Li e=C(p,w)e=0$ on this subset. This is a contradiction because, for almost every $x\in ]a,b[$, 
\begin{equation*}
\Li e(x)=\frac{1}{p(x)w(x)}\int_a^b K(x,\cdot)e>0
\end{equation*}
since $e\geq 0$ and $K(x,\cdot)>0$ on $]a,b[$.

We are left to show that $\kappa_1$ is a simple eigenvalue of $\Li$. Assume that $e_1,\,e_2\in E^2(p,w)$ are eigenvectors of $\Li $ associated to $C(p,w)$ with $\left\|e_1\right\|^2_{L^2(pw)}=\left\|e_2\right\|^2_{L^2(pw)}$. If $e_1\neq e_2$, the sets $\left\{e_1>e_2\right\}$ and $\left\{e_1<e_2\right\}$ must be non negligible because the two functions have the same norm. Thus, $e_1-e_2$ must change sign in $]a,b[$. But $e_1-e_2$ is also an eigenvector associated to $C(p,w)$. This contradicts the first part of the proof.
\end{proof}

\begin{proof}[Proof of Proposition \ref{AAA}]
We recall that $L^2(pw)=E^2(p,w)$ is separable. If $\Li$ is compact, as it is also self-adjoint and positive, (A1) and (A2) follow from Theorems VI.15 and VI.16 in \cite{Modern}. 
Theorem 4.6 in chapter 2 of \cite{Conway} tells us that, under (A1)-(A2), $\Li$ is compact. 
The same Theorem ensures that $\left\|\Li \right\|_{E^2(p,w)\rightarrow E^2(p,w)}$ is the largest eigenvalue of $\Li $. So, we get (A3) from Proposition \ref{Li continuous}. 
\end{proof}

\begin{proof}[Proof of Proposition \ref{prop:compconditiona}]
Before proceeding with the proof, we make the following remark. Since $\left\{e_i:i\in\N_0\right\}$ is a Hilbert basis of $E^2(p,w)$, any function $f\in E^2(p,w)$
can be written $f=\sum_{i=1}^{\infty}b_ie_i$ with $b_i=\E_p[f e_iw]$. Hence, we have $\left\|f\right\|^2_{E^2(p,w)}=\sum_{i=1}^{\infty}b_i^2$ and, using \eqref{klop1},
\begin{equation}
\label{klop2}
\left\|If\right\|^2_{L^2(p)}=\sum_{i,j}b_ib_j\E_p\left[Ie_iIe_j\right]=\sum_{i,j}b_ib_j\kappa_i\E_p\left[e_i
  e_j w\right]=\sum_{i=1}^{\infty} b_i^2\kappa_i. 
\end{equation}

Now, suppose that $h\in H^1_c(p,w)$ saturates $\PIpw$. As $h'\in E^2(p,w)$, we can write $h'=\sum_{i=1}^{\infty} b_i e_i$ for some $b_i\in\R$. 
By \eqref{klop2}, we have $\left\|h\right\|^2_{L^2(p)}=\left\|Ih'\right\|^2_{L^2(p)}=\sum_{i=1}^{\infty} b_i^2\kappa_i$.
Since $\kappa_1=C(p,w)$, it enables us to get
\begin{equation*}
\kappa_1\sum_{i=1}^{\infty} b_i^2=\kappa_1\left\|h'\right\|^2_{L^2(pw)}=\left\|h\right\|^2_{L^2(p)}=\sum_{i=1}^{\infty} b_i^2\kappa_i.
\end{equation*}
Hence, we have $\sum_{i=1}^{\infty} b_i^2(\kappa_1-\kappa_i)=0$, which implies $b_i=0$ as soon as $\kappa_i<\kappa_1$. By Proposition \ref{prop:eigen}, $\kappa_1$ is a simple eigenvalue of $\Li$. Thus, we have $h'=b_1e_1$. 
The inverse implication has already been proved in Proposition \ref{prop:eigen}. 
\end{proof}

\begin{proof}[Proof of Proposition \ref{C12}]
Property (A1) follows from Proposition \ref{eigen} and Corollary \ref{relation vp}. 
We check that (A2) holds. Let $\left\{v_i:i\in\N_0\right\}$ be the eigenvectors of $-\LS$ on $H^1_c(p,w)$ and set $e_i=v_i'$.  
Since $\left\{v_i:i\in\N_0\right\}$ is orthonormal in $H^1_c(p,w)$, we can compute  
\begin{equation*}
\E_p[e_ie_jw]=\E_p[v_i'v_j'w]=\delta_{ij}.
\end{equation*}
Let $e\in E^2(p,w)$. As $Ie\in H^1_c(p,w)$, we can write $Ie=\sum_{i=1}^{\infty}a_i v_i$ for some $a_i\in\R$. So, we have
\begin{equation*}
e=D I e=D\sum_{i=1}^{\infty}a_i v_i=\sum_{i=1}^{\infty}a_i Dv_i=\sum_{i=1}^{\infty}a_i e_i
\end{equation*}
where we have used the continuity of $D:H^1_c(p,w)\rightarrow E^2(p,w)$. We conclude that $\left\{e_i:i\in\N_0\right\}$ is a Hilbert basis of $E^2(p,w)$. As (A1)-(A2) hold, $\Li$ is compact by Proposition \ref{AAA}.
\end{proof}

\begin{proof}[Proof of Proposition \ref{bound_k}]
Since $\Li $ is a kernel operator with kernel $k_w\in L^2(pw\otimes pw)$, $\Li $ is a continuous Hilbert-Schmidt operator on $L^2(pw)$ by Theorem VI.23 in \cite{Modern}. Point (e) of Theorem VI.22 in \cite{Modern} ensures that every Hilbert-Schmidt operator is compact. As $L^2(pw)=E^2(p,w)$ by Proposition \ref{ID}, we reap the first statement. 

The first inequality is obvious since $C(p,w)=\kappa_1$. 
Theorems VI.22 and VI.23 in \cite{Modern} tell us that $\left\|k_w\right\|_{L^2(pw\otimes pw)}^2=\text{tr}(\Li^2)$ where tr denotes the trace of an operator. By Theorem VI.18 in \cite{Modern}, the trace of $\Li^2$ is
\begin{equation*}
\text{tr}(\Li^2)=\sum_{i\in \N}\E_p\left[e_i\Li^2 e_i\right]=\sum_{i\in \N}\E_p\left[\Li e_i\Li e_i\right]=\sum_{i\in \N}\kappa_i^2
\end{equation*}
where we have used the fact that $\Li$ is self-adjoint and (A1)-(A2).
\end{proof}

\begin{proof}[Proof of Proposition \ref{our_bound}]
  Let $f\in C^{\infty}(]a,b[)$ be such that $-(\LS f)'>0$ on $]a,b[$ and $\T(f'w)\in L^2(p)$. 
  We set $h=\T(f'w)=\LS f$. 
  As $p,w,f'\in C^2(]a,b[)$, $h$ is differentiable. For a differentiable function
  $g$ such that $\T g\in L^1(p)$, we necessarily have
  $\mathbb{E}_p[\T g] = \lim_{t\to \infty}g(t)p(t)-\lim_{t\to
    -\infty}g(t)p(t)$ so that, after some straightforward
  simplifications,
$$\Ti\T g(x)=\frac{1}{p(x)}\int_{-\infty}^x\left(\T g-\E[\T g]\right)p
=g(x)-\Phi g(x)$$
and thus
$\Ti h=\Ti \T(f'w)=f'w-\Phi (f'w)$. 
By assumption, we have $h\in L^2(p)$. 
Hence, we can use Theorem \ref{lem1} to obtain
\begin{equation*}
\inf \frac{f'w-\Phi (f'w)}{-(\LS f)'w}=\inf-\frac{\Ti h}{h'w}\leq C(p,w). 
\end{equation*} 
The upper bound on $C(p,w)$ can be deduced in the same way. 
\end{proof}

\begin{proof}[Proof of Proposition \ref{suite_mini}]
As $g_0\in E^2(p,w)$, we write $g_0=\sum_{i=1}^{\infty}a_i e_i$ for some $a_i\in\N$ by (A2). By assumption, we have $a_1=\E_p[g_0e_1w]\neq 0$. 
Using the continuity of $\Li $, (A2) and \eqref{klop2}, we compute
\begin{align*}
\frac{\left\|I\Li ^ng_0\right\|^2_{L^2(p)}}{\left\|\Li ^ng_0\right\|^2_{L^2(pw)}}
&=\frac{\left\|I\sum_{i=1}^{\infty}a_i \Li ^n e_i\right\|^2_{L^2(p)}}{\left\|\sum_{i=1}^{\infty}a_i \Li ^n e_i\right\|^2_{L^2(pw)}}\\
&=\frac{\left\|I\sum_{i=1}^{\infty}a_i \kappa_i^n e_i\right\|^2_{L^2(p)}}{\left\|\sum_{i=1}^{\infty}a_i \kappa_i^n e_i\right\|^2_{L^2(pw)}}\\
&=\frac{\sum_{i=1}^{\infty}a_i^2 \kappa_i^{2n+1}}{\sum_{i=1}^{\infty}a_i^2 \kappa_i^{2n}}
=\kappa_1\frac{\sum_{i=1}^{\infty}a_i^2 \left(\frac{\kappa_i}{\kappa_1}\right)^{2n+1}}{\sum_{i=1}^{\infty}a_i^2 \left(\frac{\kappa_i}{\kappa_1}\right)^{2n}}.
\end{align*}
By Proposition \ref{prop:eigen} and (A1), the last expression converges  to $\kappa_1\frac{a_1^2}{a_1^2}=\kappa_1=C(p,w)$ as $n\to \infty$.
\end{proof}

\begin{proof}[Proof of Proposition \ref{rate3}]
Set $g_n=A^n g_0$ for all $n\in \N$. Putting together Theorem \ref{convergence1} and \eqref{vitesse2}, we get
\begin{equation*}
C(p,w)|g_n-a_1e_1|(x)\leq \left\|k_w(x,\cdot)\right\|_{L^2(pw)}\left\|g_0-a_1e_1\right\|_{L^2(pw)}\left(\frac{\kappa_2}{C(p,w)}\right)^{n-1}=B_{n}(x).
\end{equation*}
We have
\begin{equation*}
\left|\frac{g_{n+1}(x)}{g_n(x)}-1\right|\leq \frac{|g_{n+1}-a_1e_1|(x)}{|g_n|(x)}+\frac{|g_{n}-a_1e_1|(x)}{|g_n|(x)}\leq \frac{2B_n(x)}{C(p,w)|g_n|(x)}.
\end{equation*}
Further, it holds that
\begin{align*}
|g_n|(x)= \left|a_1e_1-(a_1e_1-g_n)\right|(x)&\geq  \left|a_1e_1\right|(x)-\left|a_1e_1-g_n\right|(x)\\
&\geq \left|a_1e_1\right|(x)-\frac{B_n(x)}{C(p,w)}.
\end{align*}
We have gathered everything needed to conclude
\begin{align*}
\left|\frac{\Li ^{n+1}g_0}{\Li ^n g_0}(x)-C(p,w)\right|&=C(p,w) \left|\frac{g_{n+1}}{g_n}(x)-1\right|\\
&\leq \frac{2B_n(x)}{|g_n|(x)}\\
&\leq \frac{2B_n(x)}{\left|a_1e_1\right|(x)-C(p,w)^{-1}B_n(x)}\\
&=2\left(\frac{\left|a_1e_1\right|(x)}{B_n(x)}-\frac{1}{C(p,w)}\right)^{-1}. \qedhere
\end{align*}
\end{proof}

\section{Proofs for Section \ref{sec:examples}}
\label{sec:profex}
\begin{proof}[Proof of Example \ref{exemple uniforme}]
We want to show that \eqref{tobeproved} holds by recurrence. Setting $\Li^0 1=1$, equality \eqref{tobeproved} is true for $n=0$ as $E_0(x)=1$. Remark that 
$$
\Li\sin((2k+1)\pi x)=\frac{\sin((2k+1)\pi x)}{(2k+1)^2\pi^2}
$$
for all $k\in\N$. Using the recurrence hypothesis, \eqref{Eulerpolynomial}, and the previous equality, we can compute
\begin{align*}
\Li^{n+1} 1(x)&=\frac{(-1)^n}{(2n)!} \Li\left(E_{2n}\right)(x)\\
&= \frac{4}{\pi^{2n+1}}\sum_{i=0}^{\infty}\frac{\Li\left(\sin((2k+1)\pi x)\right)}{(2k+1)^{2n+1}}\\
&= \frac{4}{\pi^{2n+3}}\sum_{i=0}^{\infty}\frac{\sin((2k+1)\pi x)}{(2k+1)^{2n+3}}\\
&=\frac{(-1)^{n+1}}{(2n+2)!} E_{2n+2}(x).\qedhere
\end{align*}
\end{proof}

\begin{proof}[Proof of Example \ref{beta exact}]
For all $i\in\N$, we have
\begin{equation}
  \label{Lixi}
\Li x^i=\frac{x-x^{i+2}}{(i+1)(i+a+1)}.
\end{equation}
We can compute with \eqref{Lixi}
\begin{align*}
\Li e(x)&=\sum_{k=0}^{\infty}\frac{(-1)^{k}r^{2k+1}}{k!\Gamma(k+\frac{\alpha}{2}+1)}\Li x^{2k+1}\\
&=\sum_{k=0}^{\infty}\frac{(-1)^kr^{2k+1}}{k!\Gamma(k+\frac{\alpha}{2}+1)}\frac{x-x^{2k+3}}{(2k+2)(2k+\alpha+2)}\\
&=\frac{1}{4r^2}\sum_{k=0}^{\infty}\frac{(-1)^kr^{2k+3}}{(k+1)!\Gamma(k+\frac{\alpha}{2}+2)}(x-x^{2k+3})\\
&=\frac{1}{4r^2}\sum_{i=1}^{\infty}\frac{(-1)^{i-1}r^{2i+1}}{i!\Gamma(i+\frac{\alpha}{2}+1)}(x-x^{2i+1})\\
&=\frac{1}{4r^2}\left(\frac{rx}{\Gamma(\frac{\alpha}{2}+1)}+\sum_{i=1}^{\infty}\frac{(-1)^ir^{2i+1}}{i!\Gamma(i+\frac{\alpha}{2}+1)}x^{2i+1}\right)\\
&=\frac{1}{4r^2}e(x)
\end{align*}
where we have used the hypothesis $\sum_{i=0}^{\infty}\frac{(-1)^ir^{2i+1}}{i!\Gamma(i+\frac{\alpha}{2}+1)}=g(r)=0$. 
We know that $g(r_1\cdot)$ doesn't change sign between 0 and 1 since $r_1$ is the smallest positive root if $g$. Hence, $h=\int_0^{\cdot}g(rx)dx$ is monotone on $[0,1]$. 
As $h\in L^2(p)$, Theorem \ref{lem1} enables us to conclude that $C(p,1)={1}/{(4r_1^2)}$ and $h$ saturates $\PIpw$.

For information, we have in the case $\alpha=2$ that 
\begin{equation}
  \label{xRnx2}
  \Li^n x=\frac{(-1)^n}{4^n n!(n+1)!} xR_{n}(x^2)
\end{equation}
where the polynomials $R_n$ are defined recursively by $R_0(x)=1$ and
\begin{equation*}
R_n(x)=x^n-\sum_{k=0}^{n-1}\frac{C^k_n C^{k+1}_{n+1}}{n-k+1}R_k(x)
\end{equation*}
where $C^k_n$ denotes a binomial coefficient. The coefficients of $R_n$ form the $n$-th row of the inverse matrix of the Narayana triangle (see \cite{oeis:Narayana}).
To prove \eqref{xRnx2}, we proceed by recurrence. This equality is obviously true in the case $n=0$. Assume that it also holds for all $k\leq n$. First, observe that
\begin{align*}
\Li \left(xR_k(x^2)\right)&=\Li\left((-1)^k4^k k!(k+1)!\Li^kx\right)\\
&=(-1)^k4^k k!(k+1)!\Li^{k+1}x\\
&=\frac{(-1)^k4^k k!(k+1)!}{(-1)^{k+1}4^{k+1}(k+1)!(k+2)!}xR_{k+1}(x^2)\\
&=-\frac{1}{4(k+1)(k+2)}xR_{k+1}(x^2)
\end{align*}
for all $k\leq n-1$, where we have used the recurrence hypothesis twice. We can compute, using the recurrence hypothesis and the previous calculation,
\begin{align*}
\Li^{n+1} x&=\Li(\Li^nx)\\
&=\Li\left(\frac{(-1)^n}{4^n n!(n+1)!}xR_n(x^2)\right)\\
&= \frac{(-1)^n}{4^n n!(n+1)!}\left(\Li x^{2n+1}-\sum_{k=0}^{n-1}\frac{C^k_n C^{k+1}_{n+1}}{n-k+1}\Li \left(xR_k(x^2)\right)\right)\\
&=\frac{(-1)^n}{4^n n!(n+1)!}\left( \frac{x-x^{2n+3}}{(2n+2)(2n+4)}+\sum_{k=0}^{n-1}\frac{C^k_n C^{k+1}_{n+1}}{n-k+1}\frac{xR_{k+1}(x^2)}{4(k+1)(k+2)}\right)\\
&=\frac{(-1)^n}{4^{n+1} (n+1)!(n+2)!}\left(x-x^{2n+3}+\sum_{k=0}^{n-1}\frac{(n+1)!(n+2)!\,xR_{k+1}(x^2)}{(k+1)!(k+2)!\left\{(n-k)!\right\}^2(n-k+1)}\right)\\
&=\frac{(-1)^n}{4^{n+1} (n+1)!(n+2)!}\left(x-x^{2n+3}+\sum_{i=1}^{n}\frac{(n+1)!(n+2)!\,xR_{i}(x^2)}{i!(i+1)!\left\{(n+1-i)!\right\}^2(n+2-i)}\right)\\
&=\frac{(-1)^n}{4^{n+1} (n+1)!(n+2)!}\left(-x^{2n+3}+\sum_{i=0}^{n}\frac{C^{i}_{n+1}C^{i+1}_{n+2}}{(n+1-i+1)}xR_{i}(x^2)\right)\\
&=\frac{(-1)^{n+1}}{4^{n+1} (n+1)!(n+2)!}xR_{n+1}(x^2). \qedhere
\end{align*}
\end{proof}

\begin{proof}[Proof of Example \ref{gamma}]
We could prove the statements from Example \ref{gamma} with Theorem \ref{lem1} by computing $\Li e_1$. Instead, we prefer using Corollary \ref{coro LS} because the calculations are much easier. 
Let $h$ be as stated. It follows from straightforward computations that   $-\LS h={k}h /({\theta^2(k+1)^2})$. Remark that $h$ is increasing and that $h'(0)p(0)=\lim_{t\rightarrow +\infty}h'(t)p(t)=0$. When $k>1$, we have $h\in L^2(p)$. It then follows from Corollary \ref{coro LS} that $C(p,1)=\theta^2(k+1)^2k^{-1}$ and the saturating function is $h$. When $k\in ]0,1]$, we  only have $h\in L^1(p)$. So, we still get $C(p,1)\leq\theta^2(k+1)^2k^{-1}$ by Proposition \ref{our_bound} but no more can be said. 

For information, we have in the case $k=\theta=1$ that
\begin{equation}
\label{gamma11}
    \Li^n 1(x)=\sum_{k=1}^n \frac{C^k_n(2n-k)!}{n!}x^k. 
\end{equation} 
To prove this fact, we need two observations. First, we have
$$
\Li x^i=\sum_{j=1}^{i+1} \frac{i!}{j!}x^j.
$$
for all $i\in\N$. Second, it is possible to show by recurrence on $j$ that 
$$
\sum_{i=j-1}^{n}\frac{i(2n-i-1)!}{(n-i)!}=\frac{j(2n-j+1)!}{(n+1)(n-j+1)!}
$$
for all $n\in \N$ and $j=1,\ldots n+1$. With these two equalities, we can prove \eqref{gamma11} by recurrence.
\end{proof}

\begin{proof}[Proof of Example \ref{weibull}]
To see that $h$ saturates $\PIpw$, it is enough to compute $-\LS h= k^2 h/\lambda^k$ and to apply Corollary \ref{coro LS}.
\end{proof}

\end{document}